\documentclass[12pt,leqno]{article}
\usepackage{amsmath, amsthm, amsfonts, amssymb}
\usepackage{mathrsfs}
\usepackage{xcolor}
\newtheorem{thm}{Theorem}[section]
\newtheorem{cor}[thm]{Corollary}
\newtheorem{lem}[thm]{Lemma}
\newtheorem{prp}[thm]{Proposition}

\theoremstyle{definition}

\newtheorem{exa}[thm]{Example}

\newtheorem{rem}[thm]{Remark}
\newtheorem{defn}{Definition}[section]
\newcommand{\scr}[1]{\mathscr #1}

\numberwithin{equation}{section} \theoremstyle{remark}

\newcommand{\ua}{\uparrow}\def\lll{\lambda}

\title{{\bf Markov Processes and  Stochastic Extrinsic Derivative Flows
		on the Space of  Absolutely Continuous Measures}
\footnote{ Feng-Yu Wang is supported in part by the National Key R\& D Program of China (No. 2022YFA1006000, 2020YFA0712900) and NNSFC (11921001).  Panpan Ren is supported by NNSFC (12301180) and Research Centre for Nonlinear Analysis at Hong Kong PolyU.  Simon Wittmann is supported by Research Centre for Nonlinear Analysis at Hong Kong PolyU.
The authors thank Professor Tong Yang for his valuable comments, suggestions and support.}
}
 \author{{\bf  Panpan Ren$^{b)}$,  Feng-Yu Wang$^{a)}$,  Simon Wittmann$^{c)}$}\\
\footnotesize{ $^{a)}$ Center for Applied Mathematics, Tianjin University, Tianjin 300072, China}\\
\footnotesize{$^{b)}$ Department of Mathematics, City University of  Hong Kong, Tat Chee Avenue, Hong Kong,  China}\\
	\footnotesize{ $^{c)}$ Department of Mathematics,
	The Hong Kong Polytechnic University,  Yuk Chio Road, Hong Kong,  China}\\
\footnotesize{    panparen@cityu.edu.hk, wangfy@tju.edu.cn, simon.wittmann@polyu.edu.hk}}

\date{}
\begin{document}

\allowdisplaybreaks
\def\eins{\boldsymbol 1}
\def\R{\mathbb R}  \def\ff{\frac} \def\ss{\sqrt} \def\B{\mathbf B}
\def\N{\mathbb N} \def\kk{\kappa} \def\m{{\bf m}}
\def\ee{\varepsilon}\def\ddd{D^*}
\def\dd{\delta} \def\DD{\Delta} \def\vv{\varepsilon} \def\rr{\rho}
\def\<{\langle} \def\>{\rangle}
\def\nn{\nabla} \def\pp{\partial} \def\E{\mathbb E}
\def\d{\,\text{\rm{d}}} \def\bb{\beta} \def\aa{\alpha} \def\D{\scr D}
\def\si{\sigma} \def\ess{\text{\rm{ess}}}\def\s{{\bf s}}
\def\F{\scr F}
\def\e{\text{\rm{e}}} \def\ua{\underline a} \def\OO{\Omega}  \def\oo{\omega}
\def\tt{\widetilde}\def\[{\lfloor} \def\]{\rfloor}
\def\ti{\tilde}
\def\cut{\text{\rm{cut}}} \def\P{\mathbb P} \def\ifn{I_n(f^{\bigotimes n})}
\def\C{\scr C}      \def\aaa{\mathbf{r}}     \def\r{r}
\def\gap{\text{\rm{gap}}} \def\prr{\pi_{{\bf m},\varrho}}  \def\r{\mathbf r}
\def\Z{\mathbb Z} \def\vrr{\varrho}
\def\ac{\ll}
\def\lam{\lambda}
\def\L{\scr L}\def\Tt{\tt} \def\TT{\tt}\def\II{\mathbb I}
\def\M{\mathbb M}\def\Q{\mathbb Q} \def\texto{\text{o}} \def\LL{\Lambda}
\def\to{\rightarrow}\def\l{\ell}\def\iint{\int}\def\gg{\gamma}
\def\EE{\scr E} \def\W{\mathbb W}
\def\A{\scr A} \def\Lip{{\rm Lip}}\def\S{\mathbb S}
\def\B{\scr B}\def\Ent{{\rm Ent}} \def\i{{\rm i}}\def\itparallel{{\it\parallel}}
\def\g{{\mathbf g}}\def\Sect{{\mathcal Sec}}\def\T{\mathbb T}\def\BB{{\bf B}}
\def\f{\mathbf f} \def\g{\mathbf g}\def\BL{{\bf L}}  \def\BG{{\mathbb G}}
\def\Bd{{D^E}} \def\BdP{D^E_\phi} \def\Bdd{{\bf \dd}} \def\Bs{{\bf s}} \def\GA{\scr A}
\def\Bg{{\bf g}}  \def\Bdd{\psi_B} \def\supp{{\rm supp}}\def\div{{\rm div}}
\def\ddiv{{\rm div}}\def\osc{{\bf osc}}\def\1{{\bf 1}}\def\BD{\mathbb D}\def\GG{\Gamma}
\def\H{\mathbb H} \def\beg{\begin} \def\Ric{{\rm Ric}}
\def\beq{\beg{equation}}
\def\De{D^{\textit{E}}}
\def\Det{{\tilde D}^{\textit{E}}}
\def\Ce{{C}^{\textit{E},1}}
\def\mac{{\mathbb M^{\textnormal{ac}}}}
\def\pac{{\scr P^{\textnormal{ac}}}}
\def\cyl{\scr X}\def\vrr{\varrho}
 \def\lll{\lambda}
\def\sgn{\textnormal{sign}}
\def\en{\textnormal{Ent}}

\maketitle

\begin{abstract}
Let $\mac$ (resp.~$\pac$)  be the class    of  finite (resp.~probability) measures   absolutely continuous with respect to   a $\sigma$-finite  Radon measure   on a Polish space. We present a  criterion on the quasi-regularity of Dirichlet forms on $\mac$ and $\pac$ respectively,
in terms of  upper bound conditions given by  the uniform $(L^1+L^\infty)$-norm  of the extrinsic derivative,
see Theorem \ref{thm:QR}.
As applications, we construct a class   of general type Markov processes on $\mac$ and $\pac$ via   quasi-regular Dirichlet forms containing  the diffusion, jump and killing  terms,   see   Corollary \ref{3.2}.
Moreover,   stochastic extrinsic derivative flows  on $\mac$ and $\pac$  are studied by using quasi-regular Dirichlet forms,  which  in particular provide  martingale solutions to
  SDEs on these two spaces, with drifts given by the   extrinsic derivative of entropy functionals,  see Example \ref{exa:entropy} and Example \ref{4.9}.    \end{abstract}

\vskip 2cm

\tableofcontents

\section{Introduction}

To develop stochastic analysis on the space of measures, it is crucial to construct reasonable measure-valued   processes which may play the role of Brownian motion to drive  SDEs.
By constructing (quasi-)regular local Dirichlet forms in the $L^2$-space of  reference probability measures supported on singular distributions,
several type measure-valued diffusion processes have been constructed in the literature, see  \cite{KLV, ORS, RW20, Sch, Shao, Sturm, ST} and references therein.

Recently, O-U type Dirichlet forms was  introduced  in \cite{RW22} by using the integral of intrinsic derivative with respect to Gaussian type measures on $2$-Wasserstein space,
and in our previous joint work \cite{RWW24} we proved the quasi-regularity with  extensions to $p$-Wasserstein spaces  over Banach spaces, for $p\ge 1$.
From analysis point of view, it is natural to ask for diffusion processes on the space of absolutely continuous distributions, and   to  solve SDEs
for   density functions driven by these processes as noises.

In this paper, we fix a reference measure  $\lll$  on the base space, and  link the Hilbert space $L^2(\lll)$ to  the space of finite (respectively, probability) measures   by using the map
\beq\label{PSI} L^2(\lll)\ni f\mapsto   f^2\lll\, \Big(\text{respectively, } \
\ff{f^2\lll}{\lll(f^2)}\  \text{for}\  f\ne 0\Big), \end{equation}
where and in the following,  $f^2\lll$ is the   measure  with $\ff{\d(f^2\lll)}{\d\lll}= f^2.$

   It turns out that under this map   the Fr\'echet  derivative in $L^2(\lll)$ corresponds  to the extrinsic derivative of measures introduced in \cite{ORS}. This observation indicates  us to study quasi-regular Dirichlet forms and   SDEs
on the space of absolutely continuous measures,  by using  the extrinsic derivative, as a counterpart of \cite{RWW24} where the intrinsic derivative is adopted to construct quasi-regular Dirichlet forms on  Wasserstein spaces.

\

Let $(M,\rho)$ be a Polish space,
let $\M$ be the set of finite Borel measures on    $M$, and let
 $\scr P $ be the class of Borel probability measures on $M$. Both $\M$ and $\scr P$ are complete under the total variation distance
 $$\rr_{var}(\mu,\nu):= |\mu-\nu|(M),$$
 where $|\mu-\nu|:=(\mu-\nu)^++(\mu-\nu)^-$ is the total variation of $\mu-\nu.$

 For a fixed reference measure $0\ne \lam $ on $M$, which is assumed to be $\sigma$-finite and Radon,
 we denote
 $$h_\mu:=\ff{\d\mu}{\d\lam},\ \ \mu\in \M,\ \mu\ll\lam.$$
We aim to develop stochastic analysis on the space of absolutely continuous finite measures
 $$\mac:=\big\{\mu\in \M:\ \mu\ll \lam\big\},$$
 as well as the space of absolutely continuous probability measures
$$\pac:=\big\{\mu\in \scr P:\ \mu\ll \lam\big\}.$$
Both $\mac$ and $\pac$ are Polish under the total variation distance,  which  is reformulated as
\beq\label{VAR} \rr_{var}(\mu_1,\mu_2):= \lam(|h_{\mu_1}- h_{\mu_2}|),\ \ \mu_1,\mu_2\in \mac {\rm \ or\ }\pac,\end{equation}
and $\pac$ is a closed subspace of $\mac$.  Here and throughout the paper, we simply denote $$\lll(f)=\int_M f\d\lll$$ for   a measurable function $f$ such that the integral exists.

We now  recall the   (convexity)  extrinsic derivative  for functions of measures, which goes  back to \cite{ORS} where   the space of discrete measures is concerned.

 	\begin{itemize}
 		\item A continuous function $u:\M\to \R$ is called extrinsically differentiable  with extrinsic derivative $\De u$, if
 		\begin{equation*}
 			\De u(\mu)(x):=\lim_{\varepsilon\downarrow 0}\ff{u(\mu+\varepsilon\delta_x)-u(\mu)}{\varepsilon}\in\R
 		\end{equation*}
 		exists for all $(\mu,x)\in\M\times M$, where $\dd_x$ is the Dirac measure at $x$. We write $u\in C^{E,1}(\M)$, if additionally $\De u$ is continuous on $\M\times M$. The
 		set $\Ce_b(\M)$ comprises all $u\in C^{E,1}(\M)$ with
 		\begin{equation*}
 			\sup_{(\mu,x)\in\M\times M}\big\{|u(\mu)|+|\De u(\mu)(x)|\big\}<\infty.
 		\end{equation*}
 		\item A continuous function $u:\scr P\to \R$ is called extrinsically differentiable on $\scr P$ with convexity extrinsic derivative $\Det u$, if
 		\begin{equation*}
 			\Det u(\mu)(x):=\lim_{\varepsilon\downarrow 0}\ff{u((1-\varepsilon)\mu+\varepsilon\delta_x)-u(\mu)}{\varepsilon}\in\R
 		\end{equation*}
 		exists for all $(\mu,x)\in\scr P\times M$. We write $u\in \Ce_b(\scr P)$, if additionally $\Det u$ is continuous on $\scr P\times M$ with
 		\begin{equation*}
 			\sup_{(\mu,x)\in\scr P\times M}\big\{|u(\mu)|+|\Det u(\mu)(x)|\big\}<\infty.
 		\end{equation*}
 \item $\De$ and  $\Det$ are called,  respectively, the extrinsic derivative operator and the convexity extrinsic derivative operator.
 	\end{itemize}
Regarding a finite measure as the distribution of a particle system,  the (convexity) extrinsic derivative describes   the  rate   giving birth a small weighted particle at point $x$ (meanwhile killing the same weight uniformly from the system to keep the total mass),  while  the intrinsic derivative  introduced in \cite{AKR} refers to the rate of particles moving along vector fields.
The reader shall be referred to \cite{BRW}, \cite{RW21} and the references therein for applications of the intrinsic, respectively extrinsic, derivative and links of different  derivatives for functions of measures.

\

The remainder of this paper is organized as follows.

In Section 2, we establish a general criterion
on the quasi-regularity of Dirichlet forms on $\mac$ and  $\pac$    using    upper bound
conditions given by  the uniform $(L^1+L^\infty)$-norm of the (convexity) extrinsic derivative. Since the Dirac measure $\dd_x$ appearing to the definition of    extrinsic derivative  is beyond
$\mac$,  to make analysis on $\mac$ and $\pac$ with this derivative  we will choose   reasonable  classes of functions defined on the larger spaces $\M$ and $\scr P$, for instance the class of cylindrical functions, see \eqref{CL} below, or the $C_b^{E,1}$ class introduced above.

Using our new quasi-regularity  criterion,    we construct in Section 3
   a large class of Markov processes on the space of absolutely continuous measures,  and solve in
    Section 4      the following SDEs on $\mac$ and $\pac$:
$$\d \mu_t= -\De \bb(\mu_t) \d t + \d R_t,$$
where $\bb$ is a real function on $\M$ (respectively, $\scr P$), and the noise $R_t$ is an Ornstein-Uhlenbeck (O-U for short) type process on
$\mac$ (respectively, $\pac$).  Let $\LL$ be the Gaussian type invariant probability   measure of  $R_t$,
the semigroup of $\mu_t$ is symmetric with respect to the weighted Gaussian measure (or Gibbs measure with Hamiltonian $\bb$)  formulated as
$$\d\LL^\beta(\mu):=\ff{\e^{-\bb(\mu)} \d \LL(\mu)}{\LL(\e^{-\bb})}.$$
 Of particular interest is the entropy extrinsic derivative flows addressed in Example \ref{exa:entropy} and Example \ref{4.9},
 where the drift
    is the (weak)  extrinsic derivative   of the entropy functional
 $$\bb(\mu)=   \lll(h_\mu\ln h_\mu)-\mu(M).$$ Noting that this  function
   is not extrinsically differentiable  by the above definition, but the weak extrinsic derivative can be defined as an element in the Sobolev space $\D(\De)$, where $(\De,\D(\De))$ is the closure of
   $(\De, C_b^{E,1}(M))$ in the $L^2(\LL)$,    see the explanation after  \eqref{eq:gradFl}.

\section{A criterion for quasi-regularity}

We first recall the definition of quasi-regular Dirichlet form, which can be found in e.g. \cite{MR92}, then establish a   quasi-regularity criterion for   Dirichlet forms on $L^2(\mac,\LL)$ or $L^2(\pac,\LL)$,   for a reference probability measure $\LL$ on $\mac$ or $\pac$.

\subsection{Quasi-regular Dirichlet form}

Let $(E,\rr_E)$ be a Polish (or slightly more general, Lusin) space,  and let $\LL$ be a  probability measure on the Borel $\si$-algebra $\scr B(E)$.

 A   Dirichlet form  $(\EE,\D(\EE))$ on $L^2(E,\LL)$  is a densely defined   bilinear form having  the following properties (see   \cite[Definition I.4.5]{MR92}):
\beg{enumerate} \item[$\bullet$]  Positivity: $\EE(u,u)\ge 0, u\in \D(\EE),$
 \item[$\bullet$]  Coercivity:  there exists a constant $K>0$ such that
$$|\EE_1(u,v)|:= \big|\EE(u,v)+\LL(u,v)|\le K \ss{\EE_1(u,u)\EE_1(v,v)},\ \ u,v\in \D(\EE),$$
  \item[$\bullet$]  Markov property:    if $u\in \D(\EE),$ then $1\land u^+\in \D(\EE)$ and
  \beq\label{*}  \EE(u+u^+\land 1, u-u^+\land 1) \ge 0,\ \ \EE(u-u^+\land 1, u+u^+\land 1) \ge 0.\end{equation}
  \end{enumerate}
  In this paper, we only consider symmetric Dirichlet forms, i.e.
  $$\EE(u,v)=\EE(v,u),\ \ u,v\in \D(\EE).$$
In this case,    \eqref{*}   is equivalent to $\EE(u^+\land 1 )\le \EE( u),$  where we simply denote $\EE(u)=\EE(u,u),$
 and    $\D(\EE)$ is a separable Hilbert space with the inner product $\EE_1$, so that the $\EE_1$-norm
\begin{equation*}
	\EE_1(u)^{\ff{1}{2}}:=\big(\LL(u^2)+\EE(u,u)\big)^{\ff{1}{2}},\quad u\in\D(\EE)
\end{equation*} is complete.

We repeat some potential theoretic concepts for Dirichlet forms, found in \cite[Chapt.~III]{MR92}, and we always assume $\eins_E\in\D(\EE)$.
For further reading on the terminology and statements regarding the local property of Dirichlet forms we refer to \cite[Chapt.~V]{MR92} and \cite[Chapt.~I]{BH91}.
For an open set $O\subseteq E$ the $1$-capacity associated to $\EE$ is defined as
\begin{equation*}
	{\rm Cap}_1(O):=\inf \big\{\EE_1 (u,u)\,:\, u\in \D(\EE),\, u(z)\ge 1\text{ for }\LL\text{-a.e.~}z\in O\big\}
\end{equation*}
with the convention of $\inf(\emptyset):=\infty$. For an arbitrary set $A\subseteq E$, let
$${\rm Cap}_1(A):= \inf \big\{{\rm Cap}_1(O)\,:\, A\subseteq O,\, O \  \text{is\ an open set in }  E\big\}.   $$
An $\EE$-nest  (or nest for short) is a sequence of closed subsets $\{K_n\}_{n\in\N}$ of $E$ such that
$$\lim_{n\to\infty}{\rm Cap}_1(E\setminus K_n)= 0.$$
A measurable function $u: E \to\R$ is called quasi-continuous,  if there exists a nest $\{K_n\}_{n\in\N}$ such that the restriction $u|_{K_n}$ is continuous for each $n\in\N$.

A sequence $\{u_k\}_{k\in\N}$ of measurable functions is said to converge quasi-uniformly to a function $u: E\to\R$,  if there exists a nest $\{K_n\}_{n\in\N}$ such that the sequence of restricted functions $u_k|_{K_n}$, $k\in\N$, converge to $u|_{K_n}$ uniformly on $K_n$ as $k\to\infty$ for each $n\in\N$.

If a property, which an element $z\in E$ either has or doesn't, holds for all $z$ in the complement of a set $N\subseteq E$ with ${\rm Cap}_1(N)=0$, then this property is said to hold quasi-everywhere (q.-e.) on $E$.

For any measurable function $u$ on $E$, the smallest closed subset
$$\supp[u]:=\supp[|u|\LL]$$ of $E$ such that $(|u|\LL)(E\setminus\supp[u])=0,$
 where  $|u|\LL$ is the measure on $E$ with density $|u|$ with respect to $\LL$,
 is well-defined since $E$ is Polish and in particular strongly Lindelöf.

\begin{defn}
	\begin{itemize}
	\item[$(1)$] A  Dirichlet form $(\EE,\D(\EE))$ in $L^2(E,\LL)$ is called  local,  if
	\begin{equation}\label{eq:localPr}
		\EE(u,v)=0\quad\text{for }u,v\in\D(\EE)\ \text{with}\ \supp[u]\cap\supp[v]=\emptyset.
	\end{equation}
	\item[$(2)$] A continuous bilinear map $\Gamma:\D(\EE)\times\D(\EE)\to L^1(E,\LL)$ is called the square-field operator of $(\EE,\D(\EE))$, if
	$$\EE(u,v)=\int_E \GG(u,v)\d\LL,\ \ u,v\in \D(\EE),$$
	\begin{equation}\label{eq:squareF}
		\EE(uv,u)-\frac{1}{2}\EE(v,u^2)=\int_E v\Gamma(u,u)\d\LL,\quad u,v\in\D(\EE)\cap L^\infty(E,\LL).
	\end{equation}
	\end{itemize}
\end{defn}

To verify the local property or the existence of a square-field operator, it is enough to confirm \eqref{eq:localPr} respectively \eqref{eq:squareF} on a dense subset of $(\D(\EE),\EE_1^{1/2})$.

\beg{defn}[Definition IV.3.1 in \cite{MR92}]\label{def:QR} The  Dirichlet form $(\EE,\D(\EE))$ is called quasi-regular, if the following three conditions hold.
\beg{enumerate}
\item[Q1)]  There exists an  $\EE$-nest of compact sets (i.e.~the $1$-capacity is tight).
\item[Q2)]  $(\D(\EE),\EE_1^{1/2})$ has a dense subspace consisting of quasi-continuous functions.
\item[Q3)] There exists $N\subseteq M$ with ${\rm Cap}_1(N)=0$  and a
sequence $\{u_n\}_{n\ge 1} \subseteq \D(\EE)$ of   quasi-continuous functions  which
separate points in $E\setminus N$, i.e.~for any two different points $z_1,z_2\in E\setminus N$ there exists $n\in\N$ such that $u_n(z_1)\neq u_n(z_2)$.
\end{enumerate}\end{defn}

Note that    $\EE(u):=\EE(u,u)$ for  $u\in\D(\EE)$  can be extended to all $u\in L^2(E,\LL)$ by letting
\begin{equation*}
	\EE(u)  := \infty,\quad  \ \  u\in L^2(E,\LL)\setminus (\D(\EE).
\end{equation*}

In \cite{Fu} a correspondence between regular Dirichlet forms and strong Markov processes is built,
see  \cite{FOT11} for a complete theory and more references. This is extended in \cite{AM}, \cite{MR92} to the quasi-regular setting on a Lusin space $(E,\scr B(E))$. It was observed in  \cite{CMR} that a quasi-regular Dirichlet form $(\EE,\D(\EE))$ becomes regular under one-point compactification.

According to  \cite[Definitions IV.1.8,  IV.1.13,  V.1.10]{MR92},   a standard Markov process $$\mathbf M=(\Omega,\scr F,(X_t)_{t\geq 0},(\P_z)_{z\in E} )$$  with natural filtration $\{\scr F_t\}_{t\geq 0}$ is called a non-terminating diffusion process on $E$ if
$$\P_z\big(X_\cdot\in C([0,\infty),E)\big)=1\quad\text{for } z\in E.$$
It is called $\LL$-tight for a $\sigma$-finite measure $\LL$ if there exists a sequence $\{K_n\}_{n}$  of compact sets in $E$ such that stopping times $$\tau_n:=\inf\{t\ge 0: X_t\notin K_n\},\ \ n\in\N$$ satisfy
$$\P_z\big(\lim_{n\to\infty} \tau_n=\infty\big)=1\quad\text{for }\LL\text{-a.e.~}z\in E.$$
The diffusion process is called properly associated with a Dirichlet form $(\EE,\D(\EE))$ in $L^2(E,\LL)$, if for any bounded measurable function $u:E\to\R$ and $t> 0$,
\begin{equation}\label{eq:sg}
E\ni z\mapsto\E_z[u(X_t)]:=\int_{\Omega}u(X_t)\,\d\P_z
\end{equation}
is a quasi-continuous $\LL$-version of $P_t u$, where $(P_t)_{t\ge 0}$ is the associated Markov semigroup in $L^2(E,\LL)$ associated with $(\EE,\D(\EE))$.

\beg{defn}[Extension of the generator domain]
Let $(\EE,\D(\EE))$ be a Dirichlet form in $L^2(E,\LL)$.
The generator $(L,\D(L))$ of $(\EE,\D(\EE))$ extends to all
\begin{equation}\label{eq:extGen}
	u\in\D(\EE):\quad\exists f\in L^1(E,\LL):\quad -\EE(v,u)=\LL(vf)\quad \forall v\in\D(\EE)\cap L^\infty(E,\LL)
\end{equation}
by $Lu:=f$.
The linear space defined by \eqref{eq:extGen} is called the extended generator domain $\D_e(L)$ of $\EE$.
\end{defn}

Every $u\in\D(\EE)$ admits a quasi-continuous modification $\tilde u:E\to\R$ and defines a continuous additive functional of the process ${(X_t)}_{t\ge 0}$ via
$A^{[u]}_t:= \tilde u(X_t)- \tilde u(X_0)$, $t\ge 0$.
This and the existence of a decomposition
\begin{equation}\label{eq:fuku}
	A^{[u]}_t=M^{[u]}_t+N^{[u]}_t,\quad t\ge 0,
\end{equation}
into a martingale additive functional ${(M^{[u]}_t)}_{t\ge 0}$ of finite energy and an additive functional of zero energy  ${(N^{[u]}_t)}_{t\ge 0}$ is a consequence of
\cite[Thm.~5.2.2]{FOT11} and the method of regularization (see \cite[Chapt.~VI, in part.~Thm.~VI.2.5]{MR92}, \cite{CMR} and also \cite[Sect.~4]{Eb}).

\begin{cor}\label{cor:C1} Let $(\EE,\D(\EE))$ be a local quasi-regular   Dirichlet form  in $L^2(E,\LL)$ and $(L,\D(L))$ be its generator.
	We assume conservativity, i.e.~$\EE(\eins_E,\eins_E)=0$.
	 It holds:
	\beg{enumerate}
	\item[$(1)$] 	There exists a non-terminating diffusion process $\mathbf M=(\Omega,\scr F,(X_t)_{t\geq 0},(\P_z)_{z\in E} )$ on $E$ which is  properly associated with $(\EE,\D(\EE))$.  In particular,  $\LL$ is an invariant probability measure of  $\mathbf M$.
	\item[$(2)$] $\mathbf M$  solves the martingale problem for $(L,\D_e(L))$, i.e.~for $u\in\D_e(L)$, the additive functional
	\begin{equation*}
		 \tilde u (X_t)- \tilde u(X_0)-\int_0^t Lu(X_s)\d s,\quad t\geq 0,
	\end{equation*}
	is a $\{\scr F_t\}_{t}$-martingale under $\P_z$ for q.-e.~$z\in E$.
\end{enumerate}
\end{cor}
\begin{proof}
	$(1)$ The first statement follows from \cite[Thm.~V.1.11]{MR92}.
	
	$(2)$ The map in \eqref{eq:sg} uniquely determines  a strongly continuous contraction semigroup with parameter $t\ge 0$ in $L^1(E,\LL)$.
	Furthermore, the generator $(L^{(1)},\D(L^{(1)}))$ in $L^1(E,\LL)$ of the extended semigroup is an extension of $(L,\D_e(L))$, more precisely
	$$\D(L^{(1)})\cap \D(\EE)=\D_e(\EE).$$
	Now, we assume $u\in\D_e(L)$. Then $Lu\in \D_e(\EE)$ and
	the additive functional $t\mapsto \int_0^t {Lu}(X_s)\d s$
	coincides with ${(N^{[u]}_t)}_{t\ge 0}$ because of \cite[Thm.~5.2.4]{FOT11} and
	\begin{equation*}
		\E_z\Big[\int_0^t {Lu}(X_s)\d s\Big]=\E_z[\tilde u(X_t)]-\tilde u(z),\quad \LL\text{-a.e.~}z\in E.
	\end{equation*}
\end{proof}

\subsection{A quasi-regularity criterion with extrinsic derivative }

We now consider $(E,\rr_E)=  (\mac,\rho_{var})$ with $\rho_{var}$ as in \eqref{VAR}, and let $\LL$ be a Borel probability measure on $\mac$.
Given a measurable function $u$ on $E$ with $\LL(u^2)<\infty$, we simply denote by $u$ the corresponding element   in $L^2(E,\LL)$, which is the class of functions $\LL$-a.e.  equal to $u$.

 We will present   a handy criterion for the quasi-regularity of a Dirichlet form $(\EE,\D(\EE))$ in $L^2(\mac,\LL)$, which
  is formulated in terms of upper bounds for $\EE(\cdot,\cdot)$ by uniform norms of extrinsic derivatives, for the class of cylinder functions
 \begin{equation}\label{CL}
 	\scr X:=\big\{\M\ni \mu\mapsto g(\mu(f_1),\cdots, \mu(f_n)):n\in\N,\, f_i\in C_b(M),\, g\in C_b^1(\R^n)\big\}.
 \end{equation}

 Clearly, for such a cylindrical function $u$,  we have $u\in C^{E,1}(\M)$ with
 \begin{equation}\label{eq:extDer}
 	\De u(\mu)(x):= \sum_{i=1}^n (\pp_i g) (\mu(f_1),\cdots, \mu(f_n))  f_i(x),
 	\quad  (\mu,x)\in\M\times M,
 \end{equation}
 and $u\in C_b^{E,1}(\scr P)$ with
 \begin{equation}\label{eq:extDer2}
 	\Det u(\mu)(x):= \sum_{i=1}^n (\pp_i g) (\mu(f_1),\cdots, \mu(f_n))  (f_i(x)-\mu(f_i)),
 	\quad  (\mu,x)\in\scr P\times M.
 \end{equation}


\begin{thm}\label{thm:QR}  A Dirichlet form  $(\EE,\D(\EE))$ in $L^2(\mac,\LL)$ is quasi-regular, if $(\D(\EE),\EE^{1/2}_1)$ has a dense subspace consisting of  quasi-continuous functions and additionally one of the following holds true:
\begin{itemize}
	\item[$(1)$] There exists   a constant $C\in (0,\infty)$ such that
\begin{equation*}
	\EE(u,u)\leq C\Big(\sup_{\mu\in\mac}\|\De u(\mu)\|^2_{L^1(\mu)+L^\infty(\mu)}  + \| u\|_\infty^2 \Big),  \quad u\in \cyl.
\end{equation*}
\item[$(2)$] $\LL(\pac)=1$ and there exists a constant $C\in (0,\infty)$ such that
\begin{equation*}
	\EE(u,u)\leq C\Big(\sup_{\mu\in\pac}\|\Det u(\mu)\|^2_{L^\infty(\mu)}+ \| u\|_\infty^2\Big),\quad u\in \cyl.
\end{equation*}
In this case, $(\EE,\D(\EE))$ is also quasi-regular as a Dirichlet form in  $L^2(\pac,\LL)$.
\end{itemize}
\end{thm}

It is easy to see for any $u\in\cyl$,
$$\sup_{\mu\in\M}\|\Det u(\mu)\|^2_{L^\infty(\mu)}<\infty,\ \ \sup_{\mu\in\M}\|\De u(\mu)\|^2_{L^\infty(\mu)}<\infty.$$
However,   $\|\De u(\mu)\|^2_{L^1(\mu)}$ may be  unbounded over $\mu\in\mac$,   as   demonstrated by $u(\mu):=\mu(M)$ for which $\De u=1$. Hence, the condition in Theorem  \ref{thm:QR}(1)  is imposed   only for $u$ in the subclass
$$ \cyl_{1}:=\bigg\{u\in \cyl:\ \sup_{\mu\in\mac}\|\De u(\mu)\|_{L^1(\mu)}  <\infty\bigg\}.$$

To prove Theorem \ref{thm:QR}, we   present the following  lemma which ensurs property Q3) for the situations of Theorem \ref{thm:QR}, due to the continuity of cylindrical functions.

\begin{lem}\label{lem:Psep}
	  There exists a sequence  $\{u_n\}_{n\ge 1}\subset \cyl_{1}$   separating points in $\mac\ ($hence in $\pac)$.
	 \end{lem}
\begin{proof} Let $\{x_i\}_{i\ge 1}$ be dense in $E$, and for each $i\ge 1$ define
$$\rr_i(x):= \rr(x_i,x)\land \ff 12,\ \ x\in E.$$
Then the family $\{\rr_i\}_{i\ge 1}$ separates points in $E$. On the other hand,
it is easy to see that the set
$$\mathbb I:=\Big\{m\in \Z_+^{\mathbb N}:\ N_m:=\sup\{i\in\mathbb N: m_i>0\}<\infty\Big\}$$
is countable, where we denote $m=0$ and $N_m=0$  if  $m_i=0$ for all $i\ge 1$. For any $m\in \mathbb I$,   define
	 \begin{equation*}
		f_m(x):= \prod_{i=1}^{N_m}\Big(\ff 1 2 +\rr_i(x)\Big)^{m_i},\quad x\in \E,
	\end{equation*} under the convention that   $f_0=\eins_E$.   So, $\{f_m: m\in \mathbb I\}$ is closed under multiplication, and separates points in $E$ since it contains $\{\ff 1 2+\rr_i\}_{i\ge 1}$. By  \cite[Theorem~11(b)]{BK10},
	 the countable family $\{\mu\mapsto u_m(\mu):=\mu(f_m)\}_{m\in \mathbb I}$ separates points in $\pac$, and also in $\mac$ since $f_0=\eins_{E}$ separates the total mass of finite measures.
	
  To approximate $u_m$ by functions in $\bigcap_{1\le p <\infty}\cyl_{p}$,  for each $l\in \mathbb N$, let $\chi_l\in C^1_0([0,\infty))$ such that
 \beq \label{Q} \chi|_{[0,2l]^c}=0,\ \   \ \chi_l(s)=s\ \text{for}\ s\in [0,l].\end{equation}
We   define
$$u_{m,l}(\mu):=\chi_l \big(u_m(\mu) \big)=  \chi_l \big(\mu(f_m)\big), \ \ l\in\mathbb N,\ m\in \mathbb I.$$
By  $\chi_l(s)=s$ for $s\in[0,l],$ we have
$$\lim_{l\to\infty} u_{m,l} =  u_m,\ \ m\in\mathbb I.$$
Since $\{u_m: m\in \mathbb I\}$ separates points in $\mac$,
the family $\{u_{m,l}: m\in \mathbb I, l\in\mathbb N\}$ separates points in $\mac$ as well. It remains to verify that each $u_{m,l}\in  \cyl_{1}$.

By  \eqref{Q} we have
\begin{align*}
		 \|\De u_{m,l}(\mu)\|_{L^1(\mu)}^2 = \big[\chi_l'\big(\mu(f_m)\big)\big]^2\mu(f_m)^{2}
    \leq
		 \|\chi_l'\|_\infty^2\eins_{[0,2l]}\big( \mu(f_m)\big)  \mu(f_m)^{2}  <\infty,\ \ \mu\in \mac.
	\end{align*} Then the proof is finished.
\end{proof}

The next lemma extends the condition in  Theorem \ref{thm:QR} to the class of Lipschitz cylindrical functions.

\begin{lem}\label{lem:Lip}
	Let $n\in\N$, $u_1,\dots,u_n\in\cyl$ and $g:\R^n\to\R$ be a Lipschitz continuous function.
	In the situation of  Theorem \ref{thm:QR}(1)  it holds
	\begin{align*}
		&\EE(g(u_1,\dots,u_n))\\
		&\leq C\bigg(\Big\|\sum_{i=1}^n|\partial_ig|\Big\|_{\infty}^2 \max_{1\leq i\leq n}\sup_{\mu\in\mac}\|\De u_i(\mu)\|^2_{L^1(\mu)+L^\infty(\mu)} +\|g(u_1,\cdots,u_n)\|_\infty^2\bigg).
	\end{align*}
	In the situation of  Theorem \ref{thm:QR}(2)  it holds
	\begin{equation*}
		\EE(g(u_1,\dots,u_n))\leq C\bigg(\Big\|\sum_{i=1}^n|\partial_ig|\Big\|_{\infty}^2 \max_{1\leq i\leq n}\sup_{\mu\in\pac}\|\Det u_i(\mu)\|_{L^\infty(\mu)}^2 +\|g(u_1,\cdots,u_n)\|_\infty^2\bigg).
	\end{equation*}
\end{lem}
\begin{proof}
	If $g$ is continuously differentiable, then $g(u_1,\dots,u_n)\in\cyl$ and \eqref{eq:extDer} yields
	\begin{align*}
		\big\|\De(g(u_1,\dots,u_n))(\mu)\big\|_{L^1(\mu)+L^\infty(\mu)}&=
		\Big\|\sum_{i=1}^n\partial_ig(u_1,\dots,u_n)(\mu)\De u_i(\mu)\Big\|_{L^1(\mu)+L^\infty(\mu)}\\
		&\leq\sum_{i=1}^n |\partial_ig(u_1,\dots,u_n)(\mu)|\big\|\De u_i(\mu)\big\|_{L^1(\mu)+L^\infty(\mu)}
		\\&\leq \Big\|\sum_{i=1}^n|\partial_ig|\Big\|_{\infty}\Big(\max_{1\leq i\leq n}\|\De u_i(\mu)\|_{L^1(\mu)+L^\infty(\mu)}\Big),\quad\mu\in\mac.
	\end{align*}
	Analogously, \eqref{eq:extDer2} yields
	\begin{equation*}
		\big\|\Det(g(u_1,\dots,u_n))(\mu)\big\|_{\infty}\leq \Big\|\sum_{i=1}^n|\partial_ig|\Big\|_{L^\infty(\R^n,\d x)}\Big(\max_{1\leq i\leq n}\|\Det u_i(\mu)\|_{L^\infty(\mu)}\Big),\quad\mu\in\pac.
	\end{equation*}
	So, by the assumptions of Theorem \ref{thm:QR}, the claim holds for $g$ as chosen.
	The generalization to the case in which $g$ is merely a Lipschitz continuous function is achieved in a straight-forward manner by standard mollification techniques
	and an application of \cite[Lem.~II.2.12]{MR92}.
\end{proof}

\begin{proof}[Proof of Theorem \ref{thm:QR}]
	By Lemma \ref{lem:Psep},   only Q1), tightness of the $1$-capacity associated with $\EE$, needs proof.
	
	(a) We first prove  Q1) in the situation of  Theorem \ref{thm:QR}(2).
	
	Since $L^1(\lam)$ is separable, we choose  a dense sequence $\{f_i\}_{i\in\N}$   in $L^1_+(\lam):=\{f\in L^1(\lam): f\ge 0\}$.
	Noting that
	$$\lam(|f_i|)=\lam(f_i)=\sup_{\varphi\in L_+^\infty(\lam), \|\varphi\|_\infty\le 1} \lam(f_i\varphi),$$
	    there exists a sequence $\{\varphi_{i,j}\}_{j\ge 1}\subset  C_b(M),$ such that $0\le \varphi_{i,j} \leq 1$ and
	\begin{equation*}
		\sup_{j\ge 1} \lam(f_i\varphi_{i,j})=\lam(|f_i|)=\lam(f_i).
	\end{equation*} Since $\{f_i\}_{i\in\N}$   is dense in $L^1_+(\lam),$ this implies
	\begin{equation*}
		\sup_{i,j\in\N}\lam(g\varphi_{i,j})= \lam(g),\quad \ g\in L^1_+(\lam),
	\end{equation*}
	and in particular,
	\begin{equation}\label{eq:QR1}
		\rho_{var}(f\lam,f_k\lam)= \lam(|f-f_k|)=\sup_{i,j\in\N}\lam(\varphi_{i,j}(f-f_k)),\quad k\in\N,\,f\in L^1_+(\lam).
	\end{equation}
	Let
	\begin{equation}\label{eq:QRchi}
		\chi(s):=-\frac{3}{2} +\int_{-2}^{[s\lor (-2)]\land 2} \big[(t+2)^+ \land[(2-t)^+\land 1\big]\,\d t,\quad s\in\R,
	\end{equation}
	and  define
	\begin{equation}\label{eq:QRvijm}
		u_{i,j,k}(\mu):=\chi\big(\mu(\varphi_{i,j})-\lam(f_k\varphi_{i,j})\big), \quad \mu\in \mac, \,i,j,k\in\N.
	\end{equation}
	Because of \eqref{eq:QR1} and the fact that $\chi$ is a continuous, monotone increasing function, it holds
	\begin{equation}\label{eq:QR2}
		  \sup_{ i,j\in\N}u_{i,j,k}(\mu)
		= \chi(\rho_{var}(\mu,f_k\lam)),\quad \mu\in \mac,\,k\in\N.
	\end{equation}
	Since $0\le\varphi_{i,j}\le 1$ and $0\le \chi' \leq 1$, it holds
	\begin{equation*}
		\|\Det u_{i,j,k}(\mu)\|_{L^\infty(\mu)}\leq
		\big\|\varphi_{i,j}-\mu(\varphi_{i,j})\big\|_{L^\infty(\mu)}\leq 1,\quad \mu\in\pac.
	\end{equation*}
	Combining this with $\|u_{i,j,k}\|_\infty\leq\|\chi\|_\infty\le \ff 3 2$  and Lemma \ref{lem:Lip}, we obtain that in the situation of Theorem \ref{thm:QR}(2),
	\begin{equation}\label{eq:QR3}
		\EE\Big(\inf_{1\leq k\leq k'}\,\sup_{1\leq i\leq i'}\,\max_{1\leq j\leq j'}u_{i,j,k}\Big)\leq  4C,\quad i',j',k'\in\N.
	\end{equation}
	By \eqref{eq:QR2}, we obtain
	\begin{equation*}
	\lim_{i',j'\to\infty}	\inf_{1\leq k\leq k'}\,\max_{1\leq i\leq i'}\,
		\max_{1\leq j\leq j'}u_{i,j,k} =
		\inf_{1\leq k\leq k'}\chi(\rho_{var}(\,\cdot,\,f_k\lam))  \quad \text{in}\  \mac,\ \ \ \,  k\in\N.
	\end{equation*}  Since $\|\chi\|_\infty<\infty,$ by the dominated convergence theorem, the convergence holds in $L^2(\mac, \LL) (=L^2(\pac, \LL)$ since  $\LL(\pac)=1$).
	Thus,  \eqref{eq:QR3} and \cite[Lem.~I.2.12]{MR92} imply
	\begin{equation}\label{eq:QR4}
		\EE\Big(\inf_{1\leq k\leq k'}\chi\circ\rho_{var}(\,\cdot\,,f_k\lam)\Big)\leq 4C.
	\end{equation}
	By choice, $\{f_k\}_{k\geq 1}$ is dense in $L^1_+(\lam),$  and noting that $\chi(s)\downarrow 0$ as $s\downarrow 0$,  for $\mu=f\lam$ we have
	\beq\label{PO} \inf_{1\le k\le k'} \chi\circ\rho_{var}(\mu, f_k\lam)=   \chi\Big(\inf_{1\le k\le k'} \lam(|f-f_k|)\Big)\downarrow 0\ \  \text{as}\ \ k'\uparrow\infty.\end{equation} 	
	Then, 	as a consequence of \eqref{eq:QR4}, \cite[Lem.~I.2.12 \& Prop.~III.3.5]{MR92} and the monotone decreasing nature of that sequence,
	\begin{equation*}
		\lim_{k'\to\infty} \inf_{1\leq k\leq k'}\chi\circ\rho_{var}(\,\cdot\,,f_k\lam)
		 = 0\quad\text{quasi-uniformly on } \ \mac.
	\end{equation*}
	In view of $\chi(s)=s$, $s\in[-1,1]$, there exists a nest ${\{K_n\}}_{n\geq 1}$ in $\mac$ such that
	\begin{equation*}
	\lim_{k'\to \infty} \sup_{\mu\in K_n}\inf_{1\leq k\leq k'}\rho_{var}(\mu,f_k\lam)=0,\quad\  n\in\N.
	\end{equation*}
	This shows that ${\{K_n\}}_{n\geq 1}$ is a nest of totally bounded (closed) sets, i.e.~a nest of compact sets.
	Since $\pac$ is a closed set in $\mac$ and  $\LL(\pac)=1$,  we have  ${\rm Cap}_1(\mac\setminus\pac)=0$ and ${\{K_n\cap\pac\}}_{n\geq 1}$ is a nest of compact sets in $\pac$. This completes the proof.

	(b) Proof of Q1) in the situation of Theorem \ref{thm:QR}(1).
	
	We first  construct a nest ${\{K_n\}}_{n\geq 1}$ such that
	\begin{equation*}
		\sup_{\mu\in K_n}\mu(M)< \infty,\quad n\in\N.
	\end{equation*}
	To this end,   let
	\begin{equation*}
		\kappa(s):=\int_{-\infty}^st^+\land (2-t)^+\d t,\qquad  s\in\R,
	\end{equation*}
	and define
	$$w_l(\mu):=\kappa\big(\ln (1+\mu(M))-l\big),\ \ \ \mu\in\mac,\ l\in\N.$$
	Since $|\kappa'(\cdot)|\leq 1$,  it holds
	\begin{equation}\label{eq:QRwl}
		\|\De w_l(\mu)\|_{L^1(\mu)+L^\infty(\mu)}^2\leq \frac{\|\eins_\M\|_{L^1(\mu)+L^\infty(\mu)}^2}{(1+\mu(M))^2} = 1.
	\end{equation}
	So, the assumptions and $|\kappa(\cdot)|\leq 1$ imply $\sup_{l\in\N}\EE_1(w_l)\leq 2C$.
Combining this with the fact that $w_l\ge 0$ is decreasing in $l$ and $\inf_lw_l(\mu)=0$ for $\mu\in\mac$, 	 ${\{w_{l}\}}_{l\geq 1}$   converges to zero quasi-uniformly as $l\to\infty$
 by \cite[Lem.~II.2.12 \& Prop.~III.3.5]{MR92}.
	For $\mu\in\mac$ and $l\in\N$ it holds $w_l(\mu)<1/2$ if and only if $\log(1+\mu(M))< l+1$.
	So, we can choose a nest $\{K_n\}_{n\geq 1}$ such that for fixed $n\in\N$ there exists $l_n\in\N$ with
	\begin{equation}\label{*Y}
		K_n\subseteq \{\mu:w_{l_n}<1/2\}.
	\end{equation}
To show that the $1$-capacity associated with $\EE$ is tight, it suffices to fix $n\in\N$ and find another nest ${\{F^n_m\}}_{m\geq 1}$  such that each intersection $K_n\cap F^n_m$ is compact.
	Indeed, the claim would follow from
	\begin{equation*}
		\inf\big\{{\rm Cap}_1( (K_n\cap F_m^n)^c):n,m\in\N\big\}\leq
		\inf\big\{{\rm Cap}_1( K_n^c)+{\rm Cap}_1(( F_m^n)^c):n,m\in\N\big\}=0.
	\end{equation*}
	Now, let $n\in \mathbb N$ be fixed, we intend to construct the desired $\{F_m^n\}_{m\ge 1}$. From here on, we also fix $l:=l_n$   in \eqref{*Y}
	
	 	 By \eqref{eq:QRvijm}, \eqref{eq:QRwl}  and  the fact that
		 $$|\chi|\le \ff 3 2,\ \ \   |\chi'|\le 1,\ \ \ |1-w_l(\mu)|\le \eins_{[0,\exp({l+2})]}(\mu(M)),$$
		 we obtain
	\begin{align*}
		&\|\De \big((1-w_l)u_{i,j,k}\big)(\mu)\|_{L^1(\mu)+L^\infty(\mu)}\\
		&\leq \Big\|\|\chi\|_\infty\De (1-w_l)(\mu)(\cdot)+(1-w_l(\mu))\|\chi'\|_\infty\varphi_{i,j}(\cdot)\,\Big\|_{L^1(\mu)+L^\infty(\mu)}\\
		 &\leq \ff 3 2+ \e^{l+2}+1= \ff 5 2 +\e^{l+2},\quad\mu\in\mac,\, i,j,k\in\mathbb N,
	\end{align*}
	Then Lemma \ref{lem:Lip} and the fact that $|(1-w_l)u_{i,j,k}|\le \ff 3 2$ yield
	\begin{equation}\label{eq:QR6}
		\EE\Big(\inf_{1\leq k\leq k'}\,\inf_{1\leq i\leq i'}\,\max_{1\leq j\leq j'}\big(
		(1-w_l)u_{i,j,k}\big)\Big)\leq C\Big(\big(\tfrac{5}{2}+ \e^{l+2}\big)^2+\tfrac 9 4\Big)^2=:\tilde C,\quad \ i',j',k'\in\N,
	\end{equation}
	with constant $\tilde C\in(0,\infty)$ (since $l$ is fixed).
	Thanks to \eqref{eq:QR2} and Lebesgue's dominate convergence, we obtain  in $L^2(\mac,\LL)$ that
	\begin{equation*}
	\lim_{i',j'\to\infty}	\inf_{1\leq k\leq k'}\,\max_{1\leq i\leq i'}\,
		\max_{1\leq j\leq j'}\big((1-w_l)u_{i,j,k}\big)=
		\inf_{1\leq k\leq k'}\big((1-w_l)\chi(\rho_{var}(\,\cdot\,,f_k\lam))\big),\quad  \,k'\in\N.
	\end{equation*}
	So, \eqref{eq:QR6} and \cite[Lem.~I.2.12]{MR92} imply
	\begin{equation}\label{eq:QR7}
		\EE\Big(\inf_{1\leq k\leq k'}\big((1-w_l)\chi\circ\rho_{var}(\,\cdot\,,f_k\lam)\big)\Big)\leq \tilde C,\ \ k'\in\mathbb N.
	\end{equation}
	By \eqref{PO},  \eqref{eq:QR7} and \cite[Lem.~I.2.12  \& Prop.~III.3.5]{MR92},
		\begin{equation*}
	\lim_{k'\to\infty} 	\inf_{1\leq k\leq k'}\big((1-w_l)\chi(\rho_{var}(\,\cdot\,,f_k\lam))\big)=0
		 \quad\text{quasi-uniformly on }\mac.
	\end{equation*}
	In view of $\chi(s)=s$, $s\in[-1,1]$, there exists a nest in $\mac$,
	which we denote by ${\{F^n_m\}}_{m\geq 1}$, such that
	\begin{equation*}
	\lim_{k'\to\infty} 	\sup_{\mu\in F^n_m}\,\inf_{1\leq k\leq k'}\big((1-w_l(\mu))\rho_{var}(\mu,f_k\lam)\big)=0,\quad m\in\N.
	\end{equation*}
	As by choice $l=l_n$    in \eqref{*Y}, we have $1-w_{l}\ge \ff 12 $ on $K_n$,  so that this implies  	
		\begin{equation*}
			\lim_{k'\to\infty} \sup_{\mu\in K_n\cap F^n_m}\,\inf_{1\leq k\leq k'}\rho_{var}(\mu,f_k\lam)=0,\quad m\in\N.
		\end{equation*}
	Therefore,   $K_n\cap F^n_m$ is a totally bounded (closed) set, i.e.~a compact set,
	so that  ${\{F^n_m\}}_{m\geq 1}$ is as desired.
\end{proof}

\section{Quasi-regular Dirichlet forms on $\mac$ and $\pac$}\label{sec:imageF}	

In this part, we   construct general type quasi-regular Dirichlet forms on $\mac$ and $\pac$, which include the jumping, diffusion and killing parts.
To this end, we will choose reference probability measures $\LL$ as image of the map $\Psi$ defined in \eqref{PSI}, which links the Fr\'echet derivative in $L^2(M,\lam)$ and the (convexity) extrinsic derivative in $\mac$ ($\pac$).

 \subsection{Jump type quasi-regular Dirichlet forms  }

As a first application of Theorem \ref{thm:QR}, we consider jump type Dirichlet forms on $\mac$ or $\pac$ and show their quasi-regularity.

\begin{thm}\label{thm:Ju}
	\begin{enumerate}
	\item[$(1)$] Let $\LL_\M$ be a Borel probability measure on $\mac$ and $J$ be a $\sigma$-finite Radon measure on $\mac\times\mac$ such that
	$$\int_{\mac\times\mac}\big[1\land \rr_{var}(\gamma,\eta)^2\big]\d J(\gamma,\eta)<\infty,\ \ J\ll \LL_\M\times\LL_\M.$$
	Then the  bilinear form
	\begin{equation*}
		\EE^J_\M(u,v)=\int_{\mac\times\mac}(u(\gamma)-u(\eta))(v(\gamma)-v(\eta))\d J(\gamma,\eta),\quad u,v\in \Ce_b(\M),
	\end{equation*}
	is closable in $L^2(\mac,\LL_\M)$, and its closure is a quasi-regular Dirichlet form.
	
\item[$(2)$] Let $\LL_\scr P$ be a Borel probability measure on $\pac$ and $J$ is a $\sigma$-finite Radon measure on $\pac\times\pac$ such that
	$$\int_{\pac\times\pac}\big[1\land \rr_{var}(\gamma,\eta)^2\big]\d J(\gamma,\eta)<\infty,\ \ J\ll \LL_\scr P\times\LL_\scr P.$$
	Then the bilinear form
	\begin{equation*}
	\EE^J_\scr P(u,v)=\int_{\pac\times\pac}(u(\gamma)-u(\eta))(v(\gamma)-v(\eta))\d J(\gamma,\eta),\quad u,v\in \Ce_b(\scr P),
	\end{equation*}
	is closable in $L^2(\pac,\LL_\scr P)$, and its closure is a quasi-regular Dirichlet form.
	\end{enumerate}
\end{thm}
\begin{proof}
(a) By \cite[Lem.~3.3]{RW21} we get
   \begin{align*}
   	   \big|u(\gamma)-u(\eta)\big|&=\bigg|\int_0^1\int_M\De u\big((1-r)\eta-r\gamma\big)\d(\gamma-\eta)\d r\bigg|\\
	   &\leq \rho_{var}(\gamma,\eta)\sup_{\mu\in\mac}\|\De u(\mu)\|_{L^\infty(\mu)},\ \ \ u\in \Ce_b(\M).\ \gamma,\eta\in\mac,
   \end{align*}
   while due to   \cite[Lem.~3.4]{RW21} we have
   \begin{align*}
   	\big|u(\gamma)-u(\eta)\big|&=\bigg|\int_0^1\int_M\Det u\big((1-r)\eta-r\gamma\big)\d(\gamma-\eta)\d r\bigg|\\
	&\leq \rho_{var}(\gamma,\eta)\sup_{\mu\in\pac}\|\Det u(\mu)\|_{L^\infty(\mu)},\ \ \  u\in \Ce_b(\scr P),\ \gamma,\eta\in\pac.
   \end{align*}
 Therefore,
\begin{equation}\label{AA1}\beg{split} 
	&\big|u(\gamma)-u(\eta)\big|\\
	&\leq \big[1\land \rho_{var}(\gamma,\eta)\big]\sup_{\mu\in\mac}\big[2|\mu(\mu)|+ \|\De u(\mu)\|_{L^\infty(\mu)}\big],\ \ \gamma,\eta\in\mac,\    u\in \Ce_b(\M),\end{split} \end{equation}
\begin{equation}\label{AA2}\beg{split} 
	&\big|u(\gamma)-u(\eta)\big|\\
	&\leq \big[1\land\rho_{var}(\gamma,\eta)\big] \sup_{\mu\in\pac}\big[2|u(\mu)|+\|\Det u(\mu)\|_{L^\infty(\mu)}\big],\ \ \gamma,\eta\in\pac,\ u\in \Ce_b(\scr P).\end{split} 
\end{equation}

(b)	  Let ${\{u_n\}}_{n\ge 1}\subset \Ce_b(\M)$ such that
	$$
	\lim_{n\to\infty} \LL_\M(u_n^2) = \lim_{n,m\to\infty}  \EE_\M^J(u_n-u_m,u_n-u_m)=  0.
	$$
	We have to show $\EE_\M^J(u_n,u_n)\to 0$ as $n\to\infty$. By $ \lim_{n,m\to\infty} \EE_\M^J(u_n-u_m,u_n-u_m)\to 0$, we see that 
	$$w_n(\gg,\eta):= u_n(\gg)- u_n(\eta),\ \ n\ge 1$$ is a Cauchy sequence in   $L^2(\mac\times\mac, J)$. So,  there exists $w\in L^2(\mac\times\mac, J)$ such that
	\begin{equation}\label{*D}
		\lim_{n\to\infty} \int_{\mac\times\mac}\big|u_{n}(\gamma)-u_{n}(\eta)-w(\gamma,\eta)|^2\d J(\gamma,\eta)= 0.
	\end{equation}
	Moreover, by $u_n\to 0$ in $L^2(\mac,\LL_\M)$, we find a    subsequence ${\{u_{n_k}\}}_{k\ge 1}$ such that
	\begin{equation*}
		u_{n_k}(\mu) \to 0,\ \ \LL_\M\text{-a.e.~}\mu\in\mac,\,k\to\infty.
	\end{equation*}
	Hence, \eqref{*D} and $J\ll \LL_\M\times\LL_\M$ imply that $w(\gamma,\eta)=0$ for $J$-a.e.~$(\gamma,\eta)\in\mac\times\mac$, and moreover
	$$\lim_{n\to\infty} \EE_\M^J(u_n,u_n)=0,$$ so that $(\EE_\M^J, C_b^{E,1}(\M))$ is closable.
	 The Markovian property of $\EE^J_\M$ follows from a standard argument (see e.g.~\cite[Ex.I.1.2.1]{FOT11}), while the quasi-regularity is a consequence of Theorem \ref{thm:QR} and  \eqref{AA1}.
	 Therefore,   assertion (1) is proved. The second assertion can be proved similarly, by using \eqref{AA2} in place of \eqref{AA1}.
\end{proof}

\subsection{Diffusion type quasi-regular   Dirichlet forms  }
Let
$$H:=L^2(M,\lam),\ \ H_0:= H\setminus \{0\}.$$  We formulate the map $\Psi$ as follows with different image spaces:
\begin{equation}\label{MP1}
	\Psi_\M:H\ni h\mapsto h^2\lam\in\mac,
\end{equation}
\begin{equation}\label{MP2}
	\Psi_\scr P: H_0 \ni h\mapsto \frac{h^2\lam}{\lam(h^2)}\in\pac.
\end{equation}
Then, fixing a Borel probability measure $\LL_0$ on $H$ with $\LL_0(\{0\})=0$, we obtain the following reference probability measures on $\mac$ and $\pac$ respectively:
\begin{equation*}
	\LL_\M:=\LL_0\circ\Psi^{-1}_\M,\qquad\LL_\scr P:=\LL_0\circ\Psi_{\scr P}^{-1},
\end{equation*}
where  the latter can also be regarded as a probability  on $\mac$ with $\LL_\scr P(\pac)=1$, or as a probability on $\pac$. Topics related to the closability and quasi-regularity of bilinear forms are not affected by which point of view we choose, since $\pac$ is a closed subset of $\mac$.

Let
$$C_{b,2}^{E,1}(\M):=\Big\{u\in C_b^{E,1}(\M):\ \sup_{\mu\in \mac} \|\De u(\mu)\|_{L^2(\mu)}<\infty\Big\}.$$
Obviously,  $C_{b,2}^{E,1}(\M)$ contains the class of functions
$$\M\ni  \mu\mapsto  \varphi(\mu(M)) u(\mu),\ \ u\in C_b^{E,1}(\M), \varphi\in C_0^1([0,\infty)),$$
which is dense in $L^2(\mac, \LL_\M)$ for any probability measure $\LL$ on $\mac$.  Here, $C_0^1$ is the class of $C^1$ functions with compact support.

We will present conditions on $\LL_0$ such that the bilinear form
\beq\label{M}  \EE_\M(u,v):= \int_\mac  \mu\big((D^E u)(D^Ev)\big)\d\LL_\M(\mu),\ \
   u,v\in C_{b,2}^{E,1}(\M) \end{equation}
is closable in $L^2(\mac,\LL_\M)$ and its closure is a quasi-regular Dirichlet form, and correspondingly,
\beq\label{P} \EE_{\scr P} (u,v):= \int_\pac \mu\big((\tt D^E u)(\tt D^Ev)\big)\d\LL_{\scr P} (\mu),\ \ u,v\in C_b^{E,1}(\scr P)\end{equation}
is closable in $L^2(\pac,\LL_\M)$ and  its closure is a quasi-regular Dirichlet form.

More generally, we  consider   weighted bilinear forms with  a  $``$diffusion coefficient" $ A$ defined as the following two families.

\beg{enumerate}\item[$(A_\M)$]   {\bf Family $\{A(\mu)\}_{\mu\in \mac}$}. Each $A(\mu)$   is a  positively definite bounded linear operator on $L^2(M,\mu)$, such that
 for any $\eta\in C_b(\mac\times M),$  \beq\label{BB}   (A  \eta)(\mu,x):= A(\mu) \eta(\mu,\cdot)(x)  \end{equation}
is measurable in $(\mu,x)\in \mac\times M,$  and
$$ \int_{\mac}   \|A(\mu)\|_{L^2(\mu)}   \d\LL_\M(\mu)<\infty.$$
 \item[$(A_{\scr P})$] {\bf Family  $\{A(\mu)\}_{\mu\in \pac}$}.   Each $A(\mu)$ is a positively definite bounded linear operator on $L^2(M,\mu)$
 such that
for any $\eta\in C_b(\pac\times M),$  $ (A  \eta)(\mu,x)$ defined in \eqref{BB}
is measurable in $(\mu,x)\in \pac\times M,$  and
$$ \int_{\pac}   \|A(\mu)\|_{L^2(\mu)}   \d\LL_\scr P(\mu)  <\infty.$$   \end{enumerate}

 We now consider the bilinear forms
\beg{align*}& \EE_\M^A(u,v):= \int_{\mac}  \big\<A(\mu)  D^E u(\mu),  D^Ev(\mu) \big\>_{L^2(\mu)} \d\LL_\M(\mu),\ \ u,v\in C_{b,2}^{E,1}(\M),\\
& \EE_{\scr P}^A(u,v):= \int_{\pac }  \big\<A(\mu)  \tt D^E u(\mu),  \tt D^Ev(\mu) \big\>_{L^2(\mu)} \d\LL_{\scr P} (\mu),\ \ u,v\in C_b^{E,1}(\scr P).\end{align*}
 It is easy to see that for any $h,h'\in H$ with $h\ne 0$, we have $h'h^{-1}\in H:=L^2(M,\mu)$ for $\mu=h^2\lam$, and
\beq\label{A0} \beg{split} &h'\mapsto A^\M_h h':=h A\big(\Psi_\M(h)\big) (h'h^{-1}),\\
& h'\mapsto A_h^{\scr P}  h':=h A\big(\Psi_{\scr P}(h)\big) (h'h^{-1}) \end{split}  \end{equation}  give  two  families  $\{A_h^\M\}_{h\in H_0}$ and $\{A_h^{\scr P}\}_{h\in H_0}$
of positive definite bounded linear operators on $H.$

Let  $C^1(H)$  be the space of   Fr\'echet differentiable  functions $f$ on $H$   with continuous     derivative $\nabla f: H\to H$, and let
\beg{align*}&C_{b}^1(H):= \big\{f\in C^1(H):\ \|f\|_\infty+\big\|\|\nn f(\cdot)\|_H\big\|_\infty<\infty\big\},\\
& C_{b,0}^1(H)=\Big\{f\in C^1(H):\ \|f\|_\infty+\big\|\|\cdot\|_H \|\nn f(\cdot)\|_H \big\|_\infty  <\infty\Big\}. \end{align*}
 Then in    situations   $(A_{\M})$ and  $(A_{\scr P})$ we may define    bilinear forms
 \beq\label{D0} \beg{split}&\EE^{0,A}_{\M}(f,g):= \int_{H} \<A_h^\M \nn f(h), \nn g(h)\>_H\d\LL_0(h),\ \ f,g\in C_b^1(H),\\
&\EE^{0,A}_{\scr P}(f,g):= \int_{H} \<A_h^{\scr P}   \nn f(h), \nn g(h)\>_H\|h\|^2_H\d\LL_0(h),\ \ f,g\in C_{b,0}^1(H).\end{split}\end{equation}
 The following result links $\EE_\M^A$ and $\EE_{\scr P}^A$ to these two  bilinear forms   respectively.

\beg{thm}\label{TA}
  Let $H, H_0, \LL_0,   \EE^{0,A}_\M, \EE^{0,A}_\scr P, \EE_\M^A$ and $\EE_{\scr P}^A$ be introduced above.
 \beg{enumerate} \item[$(1)$] In situation $(A_\M)$, for any  $u\in C_{b,2}^{E,1}(\M),$ we have $u\circ \Psi_\M\in C_b^1(H)$ and
\beq\label{DM1} \EE_\M^A(u,v)= \ff 1 4 \EE_\M^{0,A} (u\circ\Psi_\M, v\circ\Psi_\M),\ \ \ \ u,v\in C_{b,2}^{E,1}(\M).  \end{equation}
Consequently, if $(\EE_\M^{0,A},C_b^1(H))$ is closable in $L^2(H,\LL_0)$, so is $(\EE_\M^A, C_{b,2}^{E,1}(\M))$ in $L^2(\mac,\LL_\M),$ and its closure $(\EE_{\M}^A,\D(\EE_{\M}^A))$ is a local quasi-regular Dirichlet form.

\item[$(2)$] In situation $(A_{\scr P})$, for any  $u\in C_b^{E,1}(\scr P),$ we have $u\circ \Psi_{\scr P}  \in C_{b,0}^1(H)$ and
\beq\label{DM2} \EE_{\scr P}^A(u,v)= \ff 1 4 \EE_{\scr P}^{0,A} (u\circ\Psi_{\scr P}, v\circ\Psi_{\scr P}),\ \ \ \ u,v\in C_b^{E,1}(\scr P).  \end{equation}
Consequently, if $(\EE_{\scr P}^{0,A},C_{b,0}^1(H))$ is closable in $L^2(H,\LL_0)$, so is $(\EE_{\scr P}^A, C_b^{E,1}(\scr P))$ in $L^2(\pac,\LL_{\scr P}),$
and its closure $(\EE_{\scr P}^A,\D(\EE_{\scr P}^A))$ is a local quasi-regular Dirichlet form.
\end{enumerate}
  \end{thm}

To prove Theorem \ref{TA}, we present  the following lemma on chain rules  of  the Fr\'echet derivative for composed functions with $\Psi_\M$ and $\Psi_\scr P$, which lead  to \eqref{DM1} and \eqref{DM2} respectively.

\begin{lem}\label{lem:chainR}
	\begin{itemize}
	\item[$(1)$] If $u\in \Ce_{b,2}(\mac)$, then $u\circ \Psi_{\M}\in C^1(H) $ with
	\beq\label{*P0}
		 \nabla (u\circ\Psi_\M)(h) = 2 h (\De u)(\Psi_\M(h)), \, \quad h \in H.
	\end{equation}
	In particular,
	\begin{equation*}
		\big\|\nabla (u\circ\Psi_\M)(h)\big\|_{H}=2\big\|\De u(\Psi_\M(h))\big\|_{L^2(\Psi_\M(h))},\quad h\in H.
	\end{equation*}
	\item[$(2)$] If $u\in\Ce_b(\pac)$, then
	$u\circ \Psi_\scr P\in C_{b,0}^1(H) $ with
	\beq\label{*P1}
		 \nabla (u\circ\Psi_\scr P)(h)  =  \ff{2 h}{\|h\|_H^2} \Det u(\Psi_\scr P(h)), \, \quad h\in H_0.
	\end{equation}
	In particular,
	\begin{equation*}
		\big\|\nabla (u\circ\Psi_\scr P)(h)\big\|_{H}=\frac{2\big\|\Det u(\Psi_\scr P(h))\big\|_{L^2(\Psi_\scr P(h))}}{\|h\|_{H}},\quad h\in H_0.
	\end{equation*}
	\end{itemize}
\end{lem}
\begin{proof} $(1)$ Let $u\in\Ce_{b,2}(\M)$ and $h\in H$. For $\phi\in H_0$, we set
	\begin{equation*}
		\mu_{\phi,r}:=(1-r)\Psi_\M (h)+r\Psi_\M (h+\phi),\quad r\in(0,1).
	\end{equation*}	
	Then, by  \cite[Lemma 3.2]{RW21}, it holds
	\beq\label{eq:chainR1}\beg{split}
		&(u\circ\Psi_\M)(h+ \phi)-(u\circ\Psi_\M)(h)=\int_0^1\ff{\d u(\mu_{\phi,s})}{\d s}\Big|_{s=r}\d r\\
		&=\int_0^1\d r \int_{M}\De u(\mu_{\phi,r})  \d\big(\Psi_\M(h+\phi)-\Psi_\M(h)\big) \\
		&=\int_0^1\d r \int_{M}\De u(\mu_{\phi,r})\cdot\big((h+\phi)^2-h^2\big)\d\lam.
	\end{split}\end{equation}
	Since $u\in C_b^{E,1}(\M)$, $\De u$ is bounded and
	$$
		\lim_{\|\phi\|_{H}\downarrow   0}\De u(\mu_{\phi,r}) =\De u(\Psi_\M(h)).$$
	By combining this with \eqref{eq:chainR1} and
	$$
		\lim_{\|\phi\|_{H}\downarrow   0}\ff{\|(h+\phi)^2-h^2-2h\phi\|_{L^1(\lam)}}{\|\phi\|_{H}}=0,$$ 	we may apply   the triangular inequality and the dominated convergence theorem to get
	\beg{align*}
		&\limsup_{\|\phi\|_{H}\downarrow   0} \bigg|\frac{(u\circ\Psi_\M)(h+ \phi)-(u\circ\Psi_\M)(h)- 2{\big\langle} \De u(\Psi_\M(h))h,
			\phi{\big\rangle}_{H}} {\|\phi\|_{H}} \bigg| \\
			& \leq
		\limsup_{\|\phi\|_{H}\downarrow   0}	\int_0^1\lam\bigg(|\De u(\mu_{\phi,r})|\cdot\bigg|\ff{(h+\phi)^2-h^2-2h\phi}{\|\phi\|_{H}}\bigg|\bigg)\d r\\
		&\quad +\limsup_{\|\phi\|_{H}\downarrow   0}\int_0^1\bigg|2\bigg\langle h\De u(\mu_{\phi,r})-h\De u(\Psi_\M(h)),\ff{\phi}{\|\phi\|_{H}}{\bigg\rangle}_H\bigg|\d r=0.
	\end{align*}
	So, $u\circ \Psi_\M$ is Frech\'et differentiable on $H$ and \eqref{*P0} holds. This together with the continuity of $\Psi_\M$ and  $u \in C_{b,2}^1(\mac)$ implies that
	$u\circ\Psi_\M\in C_b^1(H).$

	$(2)$ Let $u\in\Ce_b(\scr P)$ and $  h\in H_0$. For $\phi\in H_0\setminus\{-h\}$, we set
	\begin{equation*}
		\mu_{\phi,r}:=(1-r)\Psi_\scr P(h)+r\Psi_\scr P(h+\phi),\quad r\in (0,1),
	\end{equation*}	
	and
	\begin{equation*}
		\nu_{\phi,r}:=(1+r)\Psi_\scr P(h+\phi)-r\Psi_\scr P(h),\quad r\in(0,1).
	\end{equation*}	
	Then, by  \cite[Lemma 3.3]{RW21}, it holds
	\beq\label{eq:chainR4}\beg{split}
		&(u\circ\Psi_\scr P)(h+ \phi)-(u\circ\Psi_\scr P)(h)
		 =\int_0^1\ff{\d u(\mu_{s,\phi})}{\d s}\Big|_{s=r}\d r\\
		 &= \int_0^1\ff{\d u(\mu_{r,\phi}+s(\nu_{r,\phi}-\mu_{r,\phi}))}{\d s}{\Big|}_{s=0}\d r
		 =\int_0^1\d r \int_{M}\Det u(\mu_{r,\phi}) )\d(\nu_{r,\phi}-\mu_{r,\phi}) \\
		&=\int_0^1\d r \int_{M}\Det u(\mu_{r,\phi})\bigg(\ff{(h+\phi)^2}{\|h+\phi\|_H^2}-\ff{h^2}{\|h\|_H^2}\bigg)\d\lam.
	\end{split}\end{equation}
	It is easy to see that  	\begin{equation}\label{eq:chainR5}
		\lim_{\|\phi\|_{H}\downarrow   0}\ff{1}{\|\phi\|_H}\bigg\|\ff{(h+\phi)^2}{\|h+\phi\|_H^2}-\ff{h^2}{\|h\|_H^2}-
		\Big(\ff{2h\phi}{\|h\|_H^2}-\ff{2h^2\langle h,\phi\rangle_H}{\|h\|_H^4}\Big)\bigg\|_{L^1(\lam)}=0
	\end{equation}
	Besides, $u\in C_b^{E,1}(\scr P)$ implies that $\Det u$ is bounded and
	$$
		\lim_{\|\phi\|_{H}\downarrow   0}\Det u(\mu_{\phi,r}) =\Det u(\Psi_\scr P(h)).$$
	Combining this with \eqref{eq:chainR4},  \eqref{eq:chainR5} and the fact that
	\begin{equation}\label{eq:chainR7}
		\Big\langle h\Det u(\Psi_\scr P(h)),\ff{h\langle h,\phi\rangle_H}{\|h\|_H^4}\Big\rangle_H=
		\ff{\langle h,\phi\rangle_H}{\|h\|_H^2}\int_{M}\Det u\Big(\ff{h^2}{\|h\|_H^2}\lam\Big)\ff{h^2}{\|h\|_H^2}\d \lam=0,
	\end{equation}
	we may apply  the triangular inequality and the dominated convergence theorem to derive
	\begin{align*}
		&\limsup_{\|\phi\|_{H}\downarrow   0}\ff{1}{\|\phi\|_{H}} \bigg|(u\circ\Psi_\scr P)(h+ \phi)-(u\circ\Psi_\scr P)(h)-\frac{ 2{\big\langle} \Det u(\Psi_\scr P(h))h,\phi\rangle_H}{\|h\|_H^2}
			  \bigg| \\
	&\leq
		\limsup_{\|\phi\|_{H}\downarrow   0}
			\int_0^1\lam\bigg(|\Det u(\mu_{\phi,r})|\cdot\ff{1}{\|\phi\|_{H}}\bigg|\ff{(h+\phi)^2}{\|h+\phi\|_H^2}-\ff{h^2}{\|h\|_H^2}-
		\Big(\ff{2h\phi}{\|h\|_H^2}-\ff{2h^2\langle h,\phi\rangle_H}{\|h\|_H^4}\Big)\bigg|\bigg)\d r\\
		&\quad +\limsup_{\|\phi\|_{H}\downarrow   0}\int_0^1\bigg|2\bigg\langle h\Det u(\mu_{\phi,r})-h\Det u(\Psi_\scr P(h)),
		\ff{\phi}{\|h\|_H^2\|\phi\|_{H}}-\ff{h\langle h,\phi\rangle_H}{\|h\|_H^4\|\phi\|_{H}}{\bigg\rangle}_H\bigg|\d r=0.
	\end{align*}
	The   remainder of the proof is similar to case (1).
	 \end{proof}

\beg{proof}[Proof of Theorem \ref{TA}]
(1)  Let $u,v\in C_{b,2}^{E,1}(\M)$. By Lemma \ref{lem:chainR}(1),
 \eqref{A0} and \eqref{D0}, we obtain
\beg{align*}&\EE_\M^{0,A} (u\circ\Psi_\M, v\circ\Psi_\M)
 =4\int_H \big\<A(\Psi_\M(h))   D^Eu(\Psi_\M(h)),   D^E v(\Psi_\M(h))\big\>_{L^2(\Psi_\M(h))} \d\LL_0(h)\\
&=  4\int_{\mac} \big\<A(\mu)  D^Eu(\mu),   D^E v(\mu)\big\>_{L^2(\mu)} \d\LL_\M(\mu)
=4 \EE_\M^A(u,v).\end{align*}
So,   \eqref{DM1} holds.

If $(\EE_\M^{0,A},C_b^1(H))$ is closable in $L^2(H,\LL_0)$, then the closability of $(\EE_\M^A, C_{b,2}^1(\mac))$ in $L^2(\mac,\LL_\M)$ is an immediate consequence of the equalities \eqref{DM1} and
	\begin{equation}\label{eq:EEmageL2}
	\|u\|_{L^2(\mac,\LL_\M)}=\|u\circ\Psi_\M\|_{L^2(H,\LL_0)},\quad u\in L^2(\mac,\LL_\M),
	\end{equation} while the coercivity and
	the  Markovian property is inherited from $\EE_\M^{0,A}$ to $\EE_\M^A$ by   \eqref{DM1} as a consequence of \cite[Prop.~I.4.7 \& I.4.10]{MR92}. So,
	  the closure $(\EE_\M^A,\D(\EE_\M^A))$ is a symmetric Dirichlet from with
	 	\begin{equation}\label{eq:EEmageDomain}
		\EE_\M^A(u,v)=\EE_\M^{0,A}(u\circ\Psi_\M,v\circ\Psi_\M),\ \  u,v\in \D(\EE_\M^A)\subseteq \big\{u:u\circ\Psi_\M\in\D(\EE_\M^{0,A})\big\}.
	\end{equation}
	Hence,  by the continuity of $\Psi_\M$,  the local property is inherited from $\EE_\M^{0,A}$ to $\EE_\M^A$ as well.
	
	Moreover, the quasi-regularity follows since the condition in Theorem \ref{thm:QR}(1) holds for $$ C= \int_{\mac}  \|A(\mu)\|_{L^2(\mu)} \d\LL_\M(\mu).$$

	(2) Let $u,v\in\Ce_b(\pac)$. By Lemma \ref{lem:chainR}(2),  \eqref{A0} and \eqref{D0}, we obtain
\beg{align*}&\EE_{\scr P}^{0,A} (u\circ\Psi_{\scr P}, v\circ\Psi_{\scr P})= 4\int_H \ff 1 {\lam(h^2)} \big\<h A(\Psi_{\scr P}(h))  \tt D^Eu(\Psi_{\scr P}(h)), h \tt D^E v(\Psi_{\scr P}(h))\big\>_H\d\LL_0(h)\\
&=4\int_H   \big\<A(\Psi_{\scr P}(h))   \tt D^Eu(\Psi_{\scr P}(h)),   \tt D^E v(\Psi_{\scr P}(h))\big\>_{L^2(\Psi_{\scr P}(h))} \d\LL_0(h)  \\
&=  4\int_{\pac} \big\<A(\mu)   \tt D^Eu(\mu),  \tt D^E v(\mu)\big\>_{L^2(\mu)} \d\LL_{\scr P}(\mu)
=4 \EE_{\scr P}^A(u,v).\end{align*}   So, \eqref{DM2} holds.
 Then the remainder of the proof is similar to case (1).
\end{proof}

\begin{rem}\label{rem:predom}  When   $\int_{\mac} \mu(M) \d \LL_\M(\mu)<\infty$, in   situation  $(A_\M)$ we have $C_b^{E,1}(\M)\subset \D(\EE_\M^A)$
and
 \begin{equation}\label{eq:extension1}
		\EE_\M^A(u,v):= \int_{\mac}  \big\<A(\mu)  D^E u(\mu),  D^Ev(\mu) \big\>_{L^2(\mu)} \d\LL_\M(\mu),\ \ u,v\in C_{b}^{E,1}(\M).
	\end{equation}
To see this, let  $ u\in C_{b}^{E,1}(\M)$ and choose  $\varphi_n\in C^1_0([0,\infty))$ for $n\in\N$ such that
 $$\eins_{[0,n]}\leq \varphi_n\leq \eins_{[0,n+1]},\ \
	 \varphi'_n\leq \eins_{[n,n+1]}.$$  Then for each $n\ge 1$   the function
	$$
	\mu\mapsto 	u_n(\mu):=\varphi_n(\mu(M))u(\mu)$$ is in the class $C_{b,2}^{E,1}(\M),$ such that
		  \beg{align*}
		&\limsup_{n\to\infty}\int_{\mac}   \Big(|u(\mu)-u_n(\mu)|^2+\big\|\De u(\mu,\cdot)-\De u_n(\mu,\cdot)\big\|_{L^2(\mu)}^2\Big)\d\LL_\M(\mu)\\
		&\leq\limsup_{n\to\infty}\Big( \big\|\De u(\cdot,\cdot)\|_\infty^2+2\|u(\cdot)\|^2_\infty\Big)\LL_\M\big(\mu(M)\eins_{[n,\infty)}(\mu(M))\big)=0.
	\end{align*}
	Therefore,
	$u\in\D(\EE_\M^A)$ and   \eqref{eq:extension1} holds.
\end{rem}

\beg{exa}\label{exa:Ksign}
Let $K$ be a bounded, positively definite linear operator on $H$ with
\begin{equation}\label{eq:Knorm}
	k_1\|f\|_H^2\leq\langle Kf,f\rangle_H\leq k_2\|f\|_H^2,\ \ f\in H
\end{equation}
for constants $k_1,k_2\in(0,\infty),$ and let $\LL_0$ be a non-degenerate Gaussian measure on $H$, i.e.~having full topological support.
If we set
\begin{equation*}
	A(\mu)g:=h_\mu^{-\ff 1 2}K\big( g {h_\mu}^{\ff 1 2 } \big), \ \ g\in L^2(\mu),\,\mu\in\mac,
\end{equation*}
then the assumptions of Theorem \ref{TA}  are satisfied. Indeed,  by  \eqref{A0} we have
\begin{align*}
	 \langle K(\sgn(h)h'),\sgn(h)h'\rangle_H&=\langle A_h^{\M} h',h'\rangle_H,\\
	 \langle K(\sgn(h)h'),\sgn(h)h'\rangle_H&=\langle A_h^{\scr P} h',h'\rangle_H,
\end{align*}  so that
it suffices to apply \eqref{eq:Knorm} to $f:=\sgn(h)h'$ and use the closability in $L^2(H,\LL_0)$ of the Gaussian Dirichlet forms
\begin{equation*}
C_b^1(H)\times C_b^1(H)\ni (f,g)\mapsto \int_{H} \< \nn f(h), \nn g(h)\>_H\d\LL_0(h),
\end{equation*}
\begin{equation*}
C_{b,0}^1(H)\times C_{b,0}^1(H)\ni (f,g)\mapsto	\int_{H} \< \nn f(h), \nn g(h)\>_H\|h\|^2_H\d\LL_0(h),
	\end{equation*}
according to   \cite{AR} or \cite[Chapt.~II]{MR92}. The generators of these forms are explicit on dense subspace of $L^2(\LL_0)$, as remarked in Example \ref{3.6} below.
	
	Hence, by  Theorem \ref{TA} and Remark \ref{rem:predom}, the bilinear forms
	\beg{align*}& \EE_\M^A(u,v)= \int_{\mac}  \lambda\Big(\sqrt{h_\mu} D^Ev(\mu)  K\big(\sqrt {h_\mu}  D^E u(\mu)\big)  \Big) \d\LL_\M(\mu),\ \ u,v\in C_{b}^{E,1}(\M),\\
	& \EE_{\scr P}^A(u,v)= \int_{\pac }    \lambda\Big(\sqrt{h_\mu} \Det v(\mu)  K\big(\sqrt {h_\mu}  \Det u(\mu)\big)  \Big)  \d\LL_{\scr P} (\mu),\ \ u,v\in C_b^{E,1}(\scr P) \end{align*}
	are closable in $L^2(\mac,\LL_\M)$, respectively $L^2(\pac,\LL_{\scr P})$, and their closures are quasi-regular Dirichlet forms.
	
\end{exa}

\beg{exa}[{\bf O-U type processes}]  \label{3.6}  \

Let  $\LL_0 $ be  a non-degenerate Gaussian measure on $H:=L^2(M,\lll)$
with covariance operator $Q^{-1}$, which is positive definite operator on $H$ of trace class.
Let
$\LL_\M$ and $\LL_{\scr P}$ be in \eqref{MP1} and \eqref{MP2}.

Then the bilinear forms defined in \eqref{M} and \eqref{P} are closable in $L^2(\mac,\LL_\M)$ and $L^2(\pac,\LL_{\scr P} )$ respectively, and their closures
\beg{align*} & \EE_\M(u,v)=\int_{\mac}\big\<\De u(\mu), \De v(\mu)\big\>_{L^2(\mu)} \LL_\M(\d\mu),\ \ u,v\in \D(\EE_\M),\\
 &\EE_{\scr P}(u,v)=\int_{\pac}\big\<\Det u(\mu), \Det v(\mu)\big\>_{L^2(\mu)} \LL_{\scr P} (\d\mu),\ \ u,v\in \D(\EE_{\scr P})\end{align*}
are local quasi-regular   Dirichlet forms.
The associated diffusion processes via Corollary \ref{cor:C1} are called O-U type processes on $\mac$ and $\pac$ respectively.

Let $\{q_n>0\}_{n\ge 1}$ be all eigenvalues of $Q$ listed in the increasing order. Then the proofs of \cite[Corollary 4.1]{RWW24} and \cite[Theorem 3.2]{RW22} imply that
$   (\EE_\M,\D(\EE_\M) )$     has  the following properties:
\beg{enumerate} \item[ $(1)$] The following log-Sobolev inequality holds:
$$  \LL_\M(u^2\log u^2)\le \ff 2 {q_1} \EE_\M(u,u),\ \ u\in \D(\EE_\M),\ \LL_\M(u^2)=1.$$
Consequently,  the associated Markov semigroup $P_t$     converges exponentially to $\LL_\M$ in entropy:
$$   \LL_\M((P_tu)\log P_t u)\le \e^{-2q_1 t} \LL_\M (u\log u),\ \ t\ge 0,\ 0\le u, \ \LL_\M(u)=1.$$
Moreover, it is hypercontractive:
  \beg{align*}&\|P_t\|_{L^r(\LL_\M)\to L^{r_t}(\LL_\M)}:=\sup_{\|f\|_{L^r(\LL_\M)}\le 1} \|P_t f\|_{L^{r_t}(\LL_\M)}\le 1,\\
&\qquad \ \ t>0, \ r>1, \ r_t:= 1+ (r-1)\e^{2q_1 t}. \end{align*}
 \item[ $(2)$]    The generator  has purely discrete spectrum,  and the  Markov semigroup   $P_t$ has density $\{p_t\}_{t\geq 0}$ with respect to $\LL_\M$ satisfying
\beg{equation*} \int_{\mac\times \mac} p_t(\mu,\nu)^2 \,\d\LL_\M(\mu) \,\d\LL_\M(\nu)
\le  \prod_{n\in \mathbb N}  \Big(1+ \ff{2\e^{-2q_n t}}{(2q_n t)\land 1}\Big)<\infty,\ \ t>0.\end{equation*}
 \item[$(3)$]Correspondingly, if one can verify a functional inequality for the Dirichlet form $\EE_\scr P^{0, id}$ defined in \eqref{D0} for $A= id$ being the identity may, then
  $ (\EE_{\scr P},\D(\EE_{\scr P}) )$  shares the same functional inequality and consequent properties. We leave this open for now and close with a remark on the generator of $\EE_\scr P^{0, id}$.
  If $\{\varphi_i\}_{i\in\N}$ denotes an orthonormal basis of eigenvectors of $Q$,
  then the generators of the Gaussian forms $\EE_\M^{0, id}$, $\EE_\scr P^{0, id}$on $H$ are explicit on the linear space
  \begin{equation*}
  	\scr FC_b^2(H,Q):=\big\{F(\langle \varphi_1,\cdot\rangle_H,\dots,\langle \varphi_n,\cdot\rangle_H):n\in\N,F\in C^2_b(\R^n)\big\}
  \end{equation*}
  via the equations
  \begin{equation*}
  	\int_{H} \< \nn f(h), \nn g(h)\>_H\d\LL_0(h)=-\LL_0(fL^{\text{OU}}_0g)
  \end{equation*}	
  and	
  \begin{equation*}
  	\int_{H} \< \nn f(h), \nn g(h)\>_H\|h\|_H^2\d\LL_0(h)=-\int_H f(h)\big(\|h\|^2L^{\text{OU}}_0g-2\langle h,\nabla g(h)\rangle_H\big)\d\LL_0(h)
  \end{equation*}	
  with
  $$L^{\text{OU}}_0g(h):=\Delta g(h)-2\<Q\nabla g(h), h\>_H,\ \ f,g\in\scr FC_b^2(H,Q),$$  where  for $f=F(\langle \varphi_1,\cdot\rangle_H,\dots,\langle \varphi_n,\cdot\rangle_H),$
  \begin{align*}
  	\nabla g&=\sum_{i=1}^n\partial_i F(\langle \varphi_1,h\rangle_H,\dots,\langle \varphi_n,h\rangle_H)\varphi_i\\
  	\Delta g&=\sum_{i=1}^n\partial_i^2 F(\langle \varphi_1,h\rangle_H,\dots,\langle \varphi_n,h\rangle_H).
  \end{align*}
    \end{enumerate}
\end{exa}

\subsection{General type quasi-regular Dirichlet forms}

As a consequence    of Theorem \ref{thm:QR}, Theorem \ref{thm:Ju} and Theorem \ref{TA},  we have the following  result which  provides a general class of  the quasi-regular  Dirichlet forms on $\mac$ respectively $\pac$, which include the  jumping, killing and diffusion parts,   and hence generate a general class of Markov processes possibly with finite life times, according to    the theory of Dirichlet forms, see \cite{MR92}.

\begin{cor}\label{3.2}
	\begin{enumerate}
	 \item[$(1)$] In the situation of Theorem \ref{TA}(1), if $(\EE_\M^{0,A},C_b^1(H))$ is closable in $L^2(H,\LL_0)$,
	 then for any non-negative $V\in L^1(\mac,\LL_\M)$ and $J$ as in Theorem \ref{thm:Ju}(1),  the bilinear form
		\begin{multline*}
			\EE(u,v)=\int_{\mac\times\mac}(u(\gamma)-u(\eta))(v(\gamma)-v(\eta))\d J(\gamma,\eta)+\int_{\mac} uvV\d\LL_\M\\+\int_{\mac}  \big\<A(\mu)  D^E u(\mu),  D^Ev(\mu) \big\>_{L^2(\mu)} \d\LL_\M(\mu),
			\quad u,v\in C_{b,2}^{E,1}(\M),
		\end{multline*}
		is closable in $L^2(\mac,\LL_\M)$, and its  closure is a quasi-regular Dirichlet form.
		\item[$(2)$] In the situation of Theorem \ref{TA}(2), if $(\EE_{\scr P}^{0,A},C_{b,0}^1(H))$ is closable in $L^2(H,\LL_0)$, then for any non-negative $V\in L^1(\pac,\LL_\scr P)$ and $J$ as in Theorem \ref{thm:Ju}(2),  the bilinear form
		\begin{multline*}
			\EE(u,v)=\int_{\pac\times\pac}(u(\gamma)-u(\eta))(v(\gamma)-v(\eta))\d J(\gamma,\eta)+\int_{\pac} uvV\d\LL_\scr P\\+\int_{\pac }  \big\<A(\mu)  \tt D^E u(\mu),  \tt D^Ev(\mu) \big\>_{L^2(\mu)} \d\LL_{\scr P} (\mu),
			\quad u,v\in C_{b}^{E,1}(\scr P),
		\end{multline*}
		is closable in $L^2(\mac,\LL_\scr P)$, and its  closure is a quasi-regular Dirichlet form.
	\end{enumerate}
	\end{cor}
\beg{proof} Noting that the killing term in $\EE(u,u)$ is bounded  above by $\|u\|_\infty^1 \|V\|_{L^1}^2, $ according to Theorem   \ref{thm:QR}, the desired assertions follow by combining the proofs  of Theorem
\ref{thm:Ju} and  Theorem \ref{TA}.
\end{proof}

\section{Solving stochastic extrinsic derivative flows}\label{sec:DerFlow}

In this part, we introduce and solve stochastic extrinsic derivative flows driven by the O-U type processes introduced in Example \ref{3.6}.

 \subsection{Stochastic extrinsic derivative flows}

{\bf A. Stochastic extrinsic derivative flows  on $\mac$.}  Let $R_t^\M$ be the O-U type process on $\mac$ introduced in Example \ref{3.6}, which is associated with the Dirichlet form $(\EE_\M,\D(\EE_\M))$, the  closure of the bilinear form in \eqref{M}
  with $\LL_\M=G\circ\Psi_\M^{-1}$ for  a non-degenerate Gaussian measure $G$ on $H$.
  Let $(L_\M, \D(L_\M))$ be the generator of $(\EE_\M,\D(\EE_\M))$.

   For any $\bb\in \D(\EE_\M)$, we consider the stochastic process ${(\mu_t)}_{t\geq 0}$ on $\mac$ which solves  the stochastic differential equation
\begin{equation}\label{eq:gradFl}
	\d \mu_t= - \De \beta (\mu_t)+\d R_t^\M,\quad t\ge 0,
\end{equation}
where $(\De,\D(\EE_\M))$ is the closure of $(\De, C_b^{E,1}(\M))$ in $L^2(\mac,\LL_\M)$. The closability follows from that of $(\nn, C_b^1(H))$ in $L^2(H,\LL_0)$   for
 the non-degenerate Gaussian measure $\LL_0$ on $H$, together with $\LL_\M=\LL_0\circ\Psi_\M^{-1}$ and \eqref{*P0}.   Hence, for any $u\in \D(\EE_\M)$,
$$ \mu\mapsto (\De \beta)^\# u(\mu)  = \<\De \bb(\mu), \De u(\mu)\>_{L^2(\mu)}  $$
belongs to $L^1(\mac,\LL_\M)$,  and coincides with $\GG_\M(\bb,u)$, for
  $(\GG_\M,\D(\EE_\M))$  the square field of  $(\EE_\M,\D(\EE_\M))$.

To define the solution of \eqref{eq:gradFl}, let us recall the following SDE on $\R^d$:
$$\d X_t =-\nn V(X_t)\d t+\d R_t$$
where $V\in C^1(\R^d)$ such that $\mu_V(\d x):= \e^{-V(x)}N(\d x)$ is a probability measure,
$N$ is the standard Gaussian measure on $\R^d$, and $R_t$ is the O-U type process on $\R^d$ generated by
$L:= \DD- x\cdot\nn$.  Then the generator of $X_t$  is formulated as
$$L^V=L-(\nn V)\cdot\nn,$$ which is symmetric in $L^2(\R^d,\e^{-V}\d N)$ for $\LL$ being the
 invariant measure of the O-U process (i.e. the standard Gaussian measure on $\R^d$).
  So,
intuitively,  the generator of a solution to \eqref{eq:gradFl} is   defined as
 $$L_\M^\bb:= L_\M -(\De \beta)^\#,$$ which should be symmetric in $L^2(\mac,\LL_\M^\bb)$ for
 $$\LL_\M^\bb:= \ff{\e^{-\bb}\LL_\M}{\LL_M(\e^{-\bb})}.$$
 Here, we need to assume that $\e^{-\bb}\in L^1(\mac,\LL_\M).$
 Noting that when $\bb\in C_b^{E,1}(\M)$ we have $\e^{-\bb} v\in C_b^{E,1}(\M)$ for $v\in C_b^{E,1}(\M)$,
 so that by the chain rule of $\GG_\M$ and the integration by parts for $L_\M$, for any $u\in \D_e(L_\M)$ we have
\beq\label{KM} \beg{split} &\int_{\mac} \GG_\M(u,v) \e^{-\bb} \d \LL_\M =
\int_{\mac} \big[\GG_\M(u, v\e^{-\bb})- v\GG_\M(u, \e^{-\bb})\big]\d\LL_\M\\
&= -\int_{\mac} v \big[L_\M u -  \GG_\M(u,\bb)\big]\e^{-\bb} \d\LL_\M= -\int_{\mac} (v L_\M^\bb u)\e^{-\bb} \d \LL_\M.\end{split}\end{equation}
Therefore,   we define as follows the pre-domain $\D_0(L_\M^\bb)$ of the generator $L_\M^\bb$, as well as  the martingale solution of \eqref{eq:gradFl}.

\beg{defn}\label{defn:martingaleSol} Let $\bb\in \D(\EE_\M)$ with $\e^{-\bb}\in L^1(\mac,\LL_\M).$
\beg{enumerate}\item[$(1)$] We denote $u\in \D_0(L_\M^\bb),$ if $u\in C_b^{E,1}(\M)$ and there exists
$g\in L^1(\mac,\LL_\M^\bb)$ such that
$$\int_{\mac}\GG_\M(u,v) \d\LL_\M^\bb = -\int_{\mac} g v \d\LL_\M^\bb,\ \ v\in C_b^{E,1}(\M).$$
In this case we denote $g= L_\M^\bb u$.
\item[$(2)$]
A martingale solution of \eqref{eq:gradFl} is a family $\{\P_\mu\}_{\mu\in\mac}$ of probability measures on  the path space $C([0,\infty); \mac)$ such that the coordinate process
 $$\mu_t(\oo):=\oo_t,\ \ \oo\in \OO:=C([0,\infty); \mac),\ t\ge 0$$
 satisfies $\P^\mu(\mu_0=\mu)=1$  for every $\mu\in\mac$, and
 for any $u\in\D_0(L_\M^\bb)$,
 $$ u (\mu_t)- u(\mu_0)-\int_0^t L_\M^\bb u(\mu_s)\d s,\qquad t\geq 0$$
	is an $\{\scr F_t\}_{t}$-martingale under $\P_\mu$ for $\LL_\M$-a.e.~$\mu\in\mac$, where $\F_t$ is the filtration induced by   $ \mu_t$.
\end{enumerate}
\end{defn}
By \eqref{KM}, when $\bb\in C_b^{E,1}(\M)$ we have $C_b^{E,1}(\M)\cap  D_e(L_\M)\subset  \D_0(L_\M^\bb),$ and
$$L_\M^\bb u= \big(L_\M   - (D^E\bb)^\#\big) u,\ \ u\in C_b^{E,1}(\M)\cap D_e(L_\M).$$
The first assertion in the following result is a direct consequence of Corollary \ref{cor:C1}, since $\D_0(L_\M^\bb)\subset \D_e(L_\M^\bb)$ when $L_\M^\bb$ is a Dirichlet operator.

\beg{prp}\label{NP}  Let $\bb\in \D(\EE_\M)$ with $\e^{-\bb}\in L^1(\mac,\LL_\M).$
\beg{enumerate}
\item[$(1)$] If the bilinear form
\begin{equation}\label{eq:EEM}
	\EE_\M^\beta(u,v):=\int_{\mac}\langle \De u(\mu),\De v(\mu)\rangle_{L^2(\mu)}\d \LL^\beta_\M(\mu),\quad u,v\in \Ce_{b}(\M),
\end{equation} is closable in $L^2(\mac,\LL^\beta_\M)$, and its closure $(\EE_\M^\beta,\D(\EE_\M^\beta))$ is a local quasi-regular   Dirichlet form, then the associated diffusion process via Corollary \ref{cor:C1} is a martingale solution to \eqref{eq:gradFl}.
\item[$(2)$] If there exists an Hilbert space $H_G \hookrightarrow H$ (i.e. continuously embedded into $H$)  such that $G(H_G)=1$ and $\e^{-\bb\circ\Psi_\M}$ is lower semi-continuous on $H_G$, then
$(\EE_\M^\bb, C_b^{E,1}(\M))$ defined in $\eqref{eq:EEM} $ is closable, and hence provides  a martingale solution to \eqref{eq:gradFl}.\end{enumerate}
\end{prp}

\beg{proof}  It suffices to prove the second assertion.
Since $G$ is supported on $H_G$,  \begin{equation*}
\LL_0^\beta:=\ff {\e^{-\beta\circ\Psi_\M} G }{G(\e^{-\beta\circ\Psi_\M})}
\end{equation*} is a well defined   probability measure on $H_G$ such that
 $$\LL_\M^\beta=\LL_0^\beta\circ\Psi_\M^{-1}.$$  Consider the bilinear form
\begin{equation*}
\EE^0(u,v):=\ff 1 4\int_{H}\langle \nabla u(f),\nabla v(f)\rangle_{H}\d \LL^\beta_0(f),
\ \ u,v\in C_b^1(H),
\end{equation*}
in $L^2(H_{G},G)$. Since $\e^{-\bb\circ\Psi_M}$ is  lower semi-continuous on $H_G$, so is   $(0,\infty)\ni s\mapsto\e^{-\beta\circ\Psi_\M(f+sg)} $  for any $f,g\in H_{G}.$ Then the Hamza condition holds (see \cite[Sect.~2]{AR} and \cite[Chapt.~II]{MR92}).  Noting that $C_b^1(H)\subseteq C_b^1(H_{G})$ as $\|\cdot\|_H \le c\|\cdot\|_{H_G}$ holds for some constant $c>0$, this and \cite[Thm.~3.2 \& Cor.~3.6]{AR}  imply
the closability of $\EE^0$ in $L^2(H_{G},G)$. Since    $L^2(H_{G},G)=L^2(H,G)$ and the   closability doesn't depend on the particular choice of the topology in the state space,
we may apply Theorem \ref{TA}(1) together with Remark \ref{rem:predom} to conclude the proof
\end{proof}


{\bf B. Stochastic extrinsic derivative flows on $\pac$.}
 Let $\tt R_t^\scr P$ be the O-U type process on $\M^{ac}$ introduced in Example \ref{3.6}, but with
\beq\label{LLP} \LL_0:= \ff{\|\cdot\|_H^p G}{G(\|\cdot\|_H^p)}  \end{equation}
for a constant $p\ge 0$ in order to cancel possible singularities appeared in the study, where $G$ is a non-trivial Gaussian measure on $H$.
As explained in Example \ref{exa:Ksign},   $(\EE_\scr P^{0,id}, C_b^1(H))$ for $A=id$  defined in \eqref{D0} is closable, so that    Theorem \ref{TA} provides a
local quasi-regular   Dirichlet form    $(\EE_\scr P,\D(\EE_\scr P))$ in $L^2(\pac,\LL_\scr P)$, as the  closure of the bilinear form defined  in \eqref{P}.
   Let $(L_\scr P, \D(L_\scr P))$ be the generator.

As explained above for $\De$ that   $(\Det, \D(\EE_{\scr P}))$ is  the closure of $(\Det, C_b^{E,1}(\scr P))$ in $L^2(\scr P,\LL_{\scr P})$, and  for any
$\bb,u\in \D(\EE_{\scr P})$,
$$(\Det \bb)^\# u(\mu):= \<\Det\bb(\mu), \Det u(\mu)\>_{L^2(\mu)}=:\GG_\scr P(u,v)$$ gives an element   in $L^1(\pac,\LL_\scr P)$. For fixed
$\bb\in\D(\EE_{\scr P}),$ consider  the following stochastic differential equation on $\scr P^{ac}$:
\begin{equation}\label{4.2'}
	\d \mu_t= - \Det \beta (\mu_t)+\d \tt R_t^\scr P,\quad t\ge 0.
\end{equation}
 Similarly to \eqref{eq:gradFl}, the martingale solution of \eqref{4.2'} is defined as follows.

\beg{defn}\label{defn:martingaleSol2} Let $\bb\in \D(\Det)$ with $\e^{-\bb}\in L^1(\pac,\LL_\scr P).$ For $\LL_0$ in \eqref{LLP}, let
$$\LL_{\scr P}^\bb:=\ff{\e^{-\bb}\LL_{\scr P}^\bb}{\LL_{\scr P}^\bb(\e^{-\bb})},\ \ \LL_{\scr P}:=\LL_0\circ\Psi_{\scr P}^{-1}.$$
\beg{enumerate}\item[$(1)$] We denote $u\in \D_0(L_{\scr P}^\bb),$ if $u\in C_b^{E,1}(\pac)$ and there exists
$g\in L^1(\pac,\LL_{\scr P}^\bb)$ such that
$$\int_{\mac}\GG_{\scr P}(u,v) \d\LL_{\scr P}^\bb = -\int_{\pac} g v \d\LL_{\scr P}^\bb,\ \ v\in C_b^{E,1}(\scr P).$$
In this case we denote $g= L_{\scr P}^\bb u$.
\item[$(2)$]
A martingale solution of \eqref{4.2'} is a family $\{\P_\mu\}_{\mu\in\pac}$ of probability measures on  the path space $C([0,\infty); \pac)$ such that the coordinate process
 $$\mu_t(\oo):=\oo_t,\ \ \oo\in \OO:=C([0,\infty); \pac),\ t\ge 0$$
 satisfies $\P^\mu(\mu_0=\mu)=1$  for every $\mu\in\pac$, and
 for any $u\in\D_0(L_{\scr P}^\bb)$,
 $$ u (\mu_t)- u(\mu_0)-\int_0^t L_{\scr P}^\bb u(\mu_s)\d s,\qquad t\geq 0$$
	is an $\{\scr F_t\}_{t}$-martingale under $\P_\mu$ for $\LL_{\scr P}$-a.e.~$\mu\in\pac$, where $\F_t$ is the filtration induced by   $ \mu_t$.
\end{enumerate}
 \end{defn}

The following assertion can be proved similarly as Proposition \ref{NP}.

\beg{prp}\label{NP'}  Let $\bb\in \D(\Det)$ with $\e^{-\bb}\in L^1(\pac,\LL_\scr P).$
\beg{enumerate}
\item[$(1)$] If the bilinear form
$$
	\EE_{\scr P}^\beta(u,v):=\int_{\pac}\langle \Det u(\mu),\Det v(\mu)\rangle_{L^2(\mu)}\d \LL^\beta_{\scr P}(\mu),\quad u,v\in \Ce_{b}(\scr P),
$$ is closable in $L^2(\pac,\LL^\beta_{\scr P})$, and its closure $(\EE_{\scr P}^\beta,\D(\EE_{\scr P}^\beta))$ is a local quasi-regular  Dirichlet form, then the associated diffusion process via Corollary \ref{cor:C1} is a  martingale solution to \eqref{4.2'}.
\item[$(2)$] If there exists an Hilbert space $H_G$ continuously embedded into $H$ such that $G(H_G)=1$ and $\e^{-\bb\circ\Psi_\scr P}$ is lower semi-continuous on $H_G$, then
$(\EE_{\scr P}^\bb, C_b^{E,1}(\scr P))$   is closable, and hence provides  a martingale solution to \eqref{4.2'}.\end{enumerate}
\end{prp}

 \subsection{Stochastic extrinsic derivative flows on $\mac$ over $\R^d$}

In this part, we take $M=\R^d$ and  let $\lll(\d x)=\d x$ be the Lebesgue measure.
We will solve  \eqref{eq:gradFl} for a class of $\bb$ satisfying the following conditions, which  includes   the entropy functional
 $$\mu\mapsto \lll(h_\mu\log h_\mu)\ \text{for}\ \ h_\mu:=\ff{\d\mu}{\d\lll}$$ as a typical example,    see Example \ref{exa:entropy} below.
 \begin{enumerate}
	\item[$(\beta1)$] There exist constants  $k\in (0,\infty)$ and $\theta\in \big(1,\ff{d}{(d-2)^+}\big]\cap (1,\infty)$ such that for any $\mu\in \mac,$
	  \begin{equation*}
	  	\lam\big( (1+|\cdot|^{2k})h_\mu+ h_\mu^{\ff 1 2}+ h_\mu^{\theta}\big)<\infty
	  	\end{equation*}
	  implies  $\beta(\mu)<\infty$.
	\item[$(\beta2)$] There are constants $\aa_1,\aa_2\in[0,\infty)$ such that for $k$ in $(\beta1)$,
	\begin{equation*} -\beta(\mu)\leq \aa_1\mu\big(1+|\cdot|^{2k} \big)+\aa_2\lam\big(h_\mu^{\ff 12}\big),\ \ \mu\in \mac.\end{equation*}
	 	\item[$(\beta3)$] Let $k,\theta$ be as in $(\bb 1)$. If ${\{\mu_n,\mu \}}_{n\ge 1}\subset\mac$    such that  $ (1+|\cdot|^{2k})h_\mu+h_\mu^{\ff 12}+h_\mu^{\theta} \in L^1(\lll)$    and
	\begin{gather*}
		 \lim_{n\to\infty}\Big(\big\| (h_{\mu_n}-  h_\mu)(1+|\cdot|^{2k})\big\|_{L^1(\lam)}+\big\| h_{\mu_n}^{\ff 12}-  h_\mu^{\ff 12}\big\|_{L^1(\lam)}+\big\|h_{\mu_n}^{\theta}-h_\mu^\theta\big\|_{L^{1}(\lam)}\Big)=0,
	\end{gather*}
 then
	\begin{equation*}
		\e^{-\beta(\mu)}\leq\liminf_{n\to\infty}\e^{-\beta(\mu_n)}.
	\end{equation*}
 \end{enumerate}

Given  $\beta$ satisfying $(\beta1)$-$(\beta3)$, we intend to construct a Gaussian measure $\LL_0=G$ on   $H_G$ with $H_G\hookrightarrow H,$ such that
$\bb$ satisfies conditions in Proposition \ref{NP}.

On the Schwartz space of rapidly decaying functions we define the operator
\begin{equation*}
	T_\alpha f(x):=-\Delta f (x)+(|x|^2+\alpha)f(x),\quad x\in\R^d,\quad f\in \scr S(\R^d),
\end{equation*}
with parameter $\alpha\in(0,\infty)$, which corresponds to the Hamiltonian of the harmonic oscillator plus a multiple of the identity operator.
Then, $T_\alpha$ extends (uniquely) to a positive, self-adjoint operator $(T_\alpha,\D(T_\alpha))$ on $H:=L^2(\R^d,\lll)$ with eigenvalues
\begin{equation}\label{eq:eigenvalues}
	2(n_1+\dots+n_d)+d+\alpha,\quad n_1,\dots,n_d\in\N\cup\{0\}.
\end{equation}
The corresponding eigenfunctions $f_{n_1,\dots,n_d}\in\scr S(\R^d)$ for $n_1,\dots,n_d\in\N\cup\{0\}$ are given by
\begin{equation*}
	f_{n_1,\dots,n_d}(x)=Z_{n_1,\dots,n_d}^{-1}\,e^{-\ff {|x|^2}{2}}\prod_{i=1}^d\scr H_{n_i}(x_i),
\end{equation*}
where $\scr H_n$ denotes the $n$-th Hermite polynomial and $Z_{n_1,\dots,n_d}^{-1}$ the normalization constant in $H$.
This is a consequence of the fact that the Hermite functions
\begin{equation*}
	\xi_n:\R\ni s\mapsto {\big(\pi^{1/2}2^n n!\big)}^{-1/2}e^{-\ff {|s|^2}{2}}\scr H_{n}(s),\quad n\in\N\cup\{0\},
\end{equation*}
satisfy
\begin{equation*}
	-\xi_n''(s)+s^2\xi_n(s)=(2n+1)\xi_n(s),\quad s\in\R,
\end{equation*}
and form an orthonormal basis of $L^2(\R,\d s)$, see e.g.~\cite[Appendicies A.1 \& A.5]{HKPS}.

For $k\in (0,\infty),$
$(T_\alpha^{k/2 },\D(T_\alpha^{k/2 }))$ is a   positive, self-adjoint operator on $H$.
The Gaussian measure $\LL_0=G_{k,\aa}$ we choose  will be supported on  the following defined Hilbert space $H_{k,\aa}$.

\beg{defn} Let $\aa,k \in (0,\infty)$.
\beg{enumerate} \item[(1)] The linear space
$H_{k,\aa}:=\D(T_\alpha^{k/2 })$ is a Hilbert space with inner product
$$
	\langle g, f \rangle_{H_{k,\aa}}:=\big\langle T^{k/2}_\alpha g, T^{k/2}_\alpha f\big\rangle_H\quad \text{for }f,g\in H_{k,\aa}.$$
	\item[(2)] Let  $G_{k,\alpha}$ be the centered  Gaussian measure  on $H_{k,\aa}$ with covariance operator $T_\alpha^{-d'}$ for some constant $d'>d$:
	\begin{equation}\label{eq:covG}
	\int_{H_{k,\aa}}\langle f,h\rangle_{H_{k,\aa}}\langle g,h\rangle_{H_{k,\aa}}\d G_{k,\alpha}(h)=\langle f,T_\alpha^{-d'}g\rangle_{H_{k,\aa}},\quad f,g\in H_{k,\aa}.
	\end{equation} This is well-defined  since $d'>d$ implies
	\begin{equation*}
		\sum_{n_1,\dots,n_d\in\N\cup\{0\}}\bigg(\ff{1}{2(n_1+\dots+n_d)+d+\alpha}\bigg)^{d'}<\infty,
	\end{equation*} so that  $T_\alpha^{-d'}$ is a  symmetric trace class operator on $H_{k,\aa}$. \end{enumerate}

	 \end{defn}

\begin{rem}\label{re:estimate}
	 Obviously,  $H_{k,\aa}\hookrightarrow W^{k,2}(\R^d)$ continuously with
	\beq\label{re:estimate1}
		\int_{\R^d}f(x)^2(|x|^{2k}+\aa)\d\lam(x)  \leq \|f\|_{H_{k,\aa}}^2, \end{equation}
	\beq\label{re:estimate2} \int_{\R^d}|\nabla f|^2\d \lam(x) \leq\|f\|_{H_{k,\aa}}^2. \end{equation}
  \end{rem}


\begin{prp}\label{prp:Vqh}
	Let a measurable function $\beta:\mac\to\R\cup\{+\infty\}$  and constants $k\in (\ff d 2,\infty)$, $\alpha_1\in[0,\infty)$ be given such that $(\beta1)$-$(\beta3)$ hold.
    For constants $\aa \in [1,\infty)$, $d'\in (d,\infty)$ with
	\beq\label{**}  (d+\aa)^{d'}>2\aa_1,\end{equation}
	let $G_{k,\alpha}$ be the Gaussian measure on $H_{k,\aa}$ with covariance given in \eqref{eq:covG} and $\LL_\M:=G_{k,\alpha}\circ\Psi_\M^{-1}$. Then,
	  \begin{equation*}\LL_\M(\{\beta\in\R\})=1\quad\text{and}\quad 0<\LL_\M(\e^{-{\beta}})<\infty.\end{equation*}
	 Moreover,   the bilinear form $\EE_\M^{\beta}$ in \eqref{eq:EEM}  is well-defined and closable in $L^2(\mac,\LL_\M^{\beta})$, and its closure  is a local quasi-regular  Dirichlet form.
\end{prp}
\begin{proof}	
	 Noting that  $k> \ff  {d}{2}$ implies
	$$c_0:= \lll\big((1+|\cdot|^2)^{- k }\big)^{\ff 1 2}<\infty,$$
	   \eqref{**} and  H\"older's inequality  imply
	\beq \label{eq:H0L1}
		 \lll(|f|)  \leq \lll\big(f^2(1+|\cdot|^{2})^k \big)^{\ff 1 2} \lll\big((1+|\cdot|^2)^{- k }\big)^{\ff 1 2}
		  \leq c_0 \|f\|_{H_{k,\aa}},\ \ f\in H_{k,\aa}.
	 \end{equation}
	Next, by \eqref{re:estimate1},  \eqref{re:estimate2}   and  the Sobolev embedding theorem (see for instance \cite[Section 1.1.4]{Maz}), for $\theta\in (1,\ff d{(d-2)^+}]\cap (1,\infty)$, there exist    constants $\vv\in (0,1]$ and  $c>0$   such that
	 \beq\label{eq:SobEmb}\beg{split}
		&\|f\|_{L^{2\theta}(\lll)} \leq c \|\nabla f\|_{L^2(\lll)}^\vv \|f\|_{L^2(\lll)}^{1-\vv} \leq c \|f\|_{H_{k,\aa}}, \ \ f\in H_{k,\aa}.
	\end{split} \end{equation}
Combining this with  \eqref{re:estimate1}, \eqref{eq:H0L1} and $(\bb 1)$, we conclude
	\begin{equation*}
	\LL_\M(\{\beta\in\R\})\geq \LL_\M(\{\mu=\Psi_\M(f):f\in H_{k,\aa}\})=G_{k,\alpha}(H_{k,\aa}) =1.
	\end{equation*}
	In particular, $\LL_\M(\{\e^{-\beta}>0\})=1$.

	Next, we prepare to show $\e^{-\bb}\in L^1(\mac,\LL_\M)$ by recalling that $\sigma_{\text{max}}:= (d+\alpha)^{-d'}$ is the largest eigenvalue of the covariance operator
 $T_\alpha^{-d'}$ of $G_{k,\aa}$, since the spectrum of $T_\alpha$ is given by \eqref{eq:eigenvalues}. As a consequence of a general result on Gaussian measures on Hilbert spaces (see e.g.~\cite[Prop.~1.2.8]{DZ}), the function $\exp(\gamma\|\cdot\|_{H_{k,\aa}}^2)$ is $G_{k,\aa}$-integrable for any $\gamma<\sigma_{\text{max}}/2$. Then, it is easy
to conclude that also $\exp(\gamma\|\cdot\|_{H_{k,\aa}}^2+r\|\cdot\|_{H_{k,\aa}})$ is $G_{k,\aa}$-integrable for any $\gamma<\sigma_{\text{max}}/2$, $r\in\R$.
By \eqref{re:estimate1} and $(\bb 2)$,   we find a constant $c_1>0$ such that
	 \beg{align*}
		&\int_{\mac} \e^{\delta \mu(\R^d)-\bb(\mu)} \d\LL_\M(\mu) = \int_H \e^{\delta\|f\|_H^2-(\beta\circ\Psi_\M)(f) } \d G_{k,\aa}( f) \\
		&\le \int_H \e^{(\aa^{-1}\delta+\aa_1)\|f\|_{H_{k,\aa}}^2+   c_1\|f\|_{H_{k,\aa}} } \d G_{k,\aa}(f),\ \ \delta\ge 0.\end{align*}
The integral  is finite if we choose $\delta$ small enough such that		
\beq\label{CDX} \ff\delta\aa +\aa_1 <\ff {1} {2 \sigma_{\text{max}}}= \ff 1 2(d+\aa)^{d'}.\end{equation}
Such $\delta>0$ exists due to \eqref{**}.

So,  \begin{equation*}
\LL_0^\beta:=\ff {\e^{-\beta\circ\Psi_\M} G_{k,\alpha}}{G_{k,\alpha}(\e^{-\beta\circ\Psi_\M})}
\end{equation*} is a well defined   probability measure on $H$, with support on $H_{k,\aa}$,   which satisfies
 $$\LL_\M^\beta=\LL_0^\beta\circ\Psi_\M^{-1}$$  and
\beq\label{eq:integrability} \int_H \e^{\delta \|\cdot \|_H^2} \d \LL_0^\bb=  \int_{\mac}\e^{\delta \mu(\R^d)} \d\LL_\M^\bb (\mu)<\infty\ \text{for \ some \ } \delta>0.\end{equation}

 We now consider the bilinear form
\begin{equation*}
\EE^0(u,v):=\ff 1 4\int_{H}\langle \nabla u(f),\nabla v(f)\rangle_{H}\d \LL^\beta_0(f),
\ \ u,v\in C_b^1(H),
\end{equation*}
in $L^2(H_{k,\aa},G_{k,a})$. By \eqref{re:estimate1},  \eqref{eq:H0L1} and \eqref{eq:SobEmb},
if $f_n\to f$ in $H_{k,\aa}$, then  $\mu_n:=f_n^2\lll$ and $\mu:= f^2\lll$ satisfy
\beg{align*}&\lll\big(|h_{\mu_n}-h_\mu|(1+|\cdot|^{2k})\big)=\lll\big(|f_n^2-f^2|(1+|\cdot|^{2k})\big)\le \|f_n-f\|_{H_{k,\aa}} \big(\|f_n\|_{H_{k,\aa}}+\|f\|_{H_{k,\aa}}\big)\to 0,\\
& \|h_{\mu_n}^{\ff 1 2}- h_{\mu}^{\ff 1 2}\|_{L^1(\lll)} =\lll(|f_n-f|)\to 0,\\
&\|h_{\mu_n}^{\theta}- h_{\mu}^{\theta}\|_{L^1(\lll)} = \lll(|f_n^{2\theta}- f^{2\theta} |)
  \le 2\theta \lll\big(|f_n-f|(|f_n|^{2\theta-1}+|f|^{2\theta-1})\big)]\\
  &\quad \le 2\theta \|f_n-f\|_{L^{2\theta}(\lll)} \big(\|f_n\|_{L^{2\theta}(\lll)}^{2\theta-1}+ \|f\|_{L^{2\theta}(\lll)}^{2\theta-1}\big)\to 0.\end{align*}
Therefore,  $(\beta3)$ implies
$$  (\beta\circ\Psi_\M)(f_n)= \bb(\mu_n)\to \bb(\mu)=  (\beta\circ\Psi_\M)(f_n),$$  so that
$\beta\circ\Psi_\M$ is continuous   on $H_{k,\aa}$.
Then, as explained in the proof of Proposition \ref{NP}(2),  $(\EE^0,C_b^1(H))$ is closable in $L^2(H_{k,\aa},G_{k,\alpha})$, and  we can apply Theorem \ref{TA}(1) together with Remark \ref{rem:predom} to conclude the proof.
\end{proof}

By  Propositions \ref{NP} and \ref{prp:Vqh}, to show that the diffusion process associated with $(\EE_\M^\bb, \D(\EE_\M^\bb))$ is a martingale solution to \eqref{eq:gradFl}, it remains to verify
$\bb\in \D(\EE_\M). $    This is confirmed in the next result for $\bb$ of type
 \begin{equation}\label{eq:betaq}
		\beta_{q,\phi}(\mu):=\beg{cases} \int_{\R^d}\big(q\circ h_\mu+\phi h_\mu\big)\d\lam,\ &\text{if}\ q\circ h_\mu+\phi h_\mu \in L^1(\lll),\\
		\infty,  &\text{otherwise,} \end{cases}
	\end{equation}
for $ \mu\in \mac,$ where   $q\in C^1((0,\infty))\cap C_b([0,\infty))$, and $\phi\in C(\R^d)$   satisfies
\beq\label{PH} |\phi|\le \min\big\{\aa_1 (1+|\cdot|^{2k}),\ K (1+|\cdot|^{k})\big\}\end{equation}
for some constants  $K,\aa_1>0$. Here, $K$ can be chosen freely, $\aa_1$ will be used to verify condition $(\bb_2)$ and $k>0$ is the same constant as in $H_{k,\aa}$ with $\alpha$ depending on $\alpha_1$ as in \eqref{**}.

\begin{prp}\label{thm:beta} Let $\beta_{q,\phi}$ be   in \eqref{eq:betaq} for some $q\in C^1((0,\infty))\cap C_b([0,\infty))$ and   $\phi\in C(\R^d)$ satisfying \eqref{PH} for some constants
$K,k,\aa_1\in (0,\infty).$
Let  $\theta\in \big(1,\ff d{(d-2)^+}]\cap (1,\infty),$   and let  $ \aa \in (0,\infty)$   such that    \eqref{**} holds.
	 \begin{enumerate}
		\item[$(1)$] $\bb=\bb_{q,\phi}$ satisfies $(\bb 1)$-$(\bb 3)$ for the given constants $k,\theta,\aa_1$ together with some other constant $\aa_2\in(0,\infty)$, if  there exists $c \in (0,\infty)$ such that
	\beq\label{eq:qphi}\beg{split}
		&-c  s^{\ff 1 2}  \leq q(s)\leq c\big(1  +s^{\theta} \big),\quad s\in[0,\infty),\ \\
		&\sup_{s>0 } \ff {|q'(s)|}{ s^{-\ff 1 2}+ s^{\theta-1} }<\infty. \end{split}\end{equation}
	 \item[$(2)$] If, in addition to \eqref{**} and \eqref{eq:qphi},
	\begin{equation}\label{eq:qphi2}
		  \sup_{s>0 } \ff {|q'(s^2)|}{ s^{-\ff 1 2}+ s^{\theta-1} }<\infty,
		 \end{equation}
	  then $\beta=\bb_{q,\phi}\in\D(\EE_\M)$ and
	\begin{equation}\label{eq:Dbeta}
		 \De\beta (\mu) =  q'\circ h_\mu+\phi,\quad \,\LL_\M\text{-a.e.~}\mu\in\mac.
	\end{equation}
	\end{enumerate}
\end{prp}
\begin{proof} By \eqref{re:estimate1}, \eqref{PH} and \eqref{eq:qphi},
	  Properties $(\beta1)$ and $(\beta2)$ are obvious. To verify $(\beta3)$, we find some constant $c_1\in(0,\infty)$ such that
	  	\begin{align*}
	 |q(r)-q(t)|\leq C\int_{r}^{t}\big(s^{-\ff 1 2}+ s^{\theta-1}\big)\d s\leq c_1  \Big[2 (t^{\ff 1 2}-  r^{\ff 1 2})+\theta^{-1} (t^{\theta}-r^{\theta} )\Big],\ \ 0\le r\le t<\infty.\end{align*}
	  Hence, for any $\mu,\nu\in\mac$ with $h_\mu,h_\nu\in L^{\theta}(\lam)$ and
	 $   h_\mu^{\ff 1 2},  h_\nu^{\ff 1 2} \in L^{1}(\lam)$,
	there exists a constant $c_2 \in(0,\infty)$ such that
	\begin{equation*}
		|\beta(\mu)-\beta(\nu)|\leq c_2\big(\big\|  (h_\mu-  h_\nu)(1+|\,\cdot\|^{2k})\big\|_{L^{1}(\lam)}+\big\|  h_\mu^{\ff 12}-  h_\nu^{\ff 12}\big\|_{L^{1}(\lam)}+\|h_\mu^\theta-h_\nu^\theta\|_{L^{1}(\lam)} \big),
	\end{equation*}
	so that  $(\beta3)$ holds.  Therefore, assertion (1) is proved.
	
	Below,  we split the  proof of assertion (2) into three steps by approximations.
	
	(a) Let $a< b\in\R^d$ (i.e.~$a_i<b_i$) and let  $c\in \R$. Consider
	\begin{equation*}
		\beta_{abc}(\mu):=q\Big(\ff{\mu([a,b))}{\lam([a,b))}\Big)\lam([a,b))+c \mu([a,b)),\quad \mu\in\mac.
	\end{equation*}
	Then $\bb_{abc}$ is extrinsically differentiable with
	$$\De \bb_{abc}(\mu)= 1_{[a,b)}\Big[ q'\Big(\ff{\mu([a,b))}{\lam([a,b))}\Big)+ c\Big].$$
	We prove   $\bb\in \D(\EE_\M)$,   we approximate $\bb_{abc}$ by functions in $C_b^{E,1}(\M)$ as follows.
	
For small $\vv>0$ we choose ${\{\tau_\varepsilon\}}_{\varepsilon> 0}\subset C_c(\R^d)$ such that $\eins_{[a,b]}\leq \tau_\varepsilon\leq {\eins_{[a-\varepsilon,b+\varepsilon]}}$, where $a-\varepsilon:={(a_i-\varepsilon)}_{i=1}^d$ and $b-\varepsilon:={(b_i-\varepsilon)}_{i=1}^d$. Obviously,
\begin{equation*}
	\beta_\varepsilon(\mu):=q\Big(\ff{\mu(\tau_\varepsilon)}{\lam([a,b))}\Big)\lam([a,b))+\ff{c \mu( \tau_\varepsilon)}{1+\vv  \mu( \tau_\varepsilon)},\quad \mu\in\mac,
\end{equation*}
is a member of $\Ce_{b}(\M)\subset \D(\EE_\M)$ with
$$D^E \bb_\vv(\mu)= q'\Big(\ff{\mu(\tau_\varepsilon)}{\lam([a,b))}\Big)\tau_\vv + \ff{c\tau_\vv}{(1+\vv   \mu(\phi \tau_\varepsilon))^2}.$$
Then
\begin{equation}\label{eq:tauE}
	\beta_\varepsilon(\mu)\overset{\varepsilon\downarrow 0}{\rightarrow}\beta(\mu),\quad
	\Big\|\eins_{[a,b)}(x)\Big[q'\Big(\ff{\mu([a,b))}{\lam([a,b))}\Big)+c\Big]-\De\beta_\varepsilon(\mu,x)\Big\|_{L^2(\mu)}\overset{\varepsilon\downarrow 0}{\rightarrow}\beta(\mu),\quad\mu\in\mac,
\end{equation} and by \eqref{eq:H0L1} and \eqref{eq:SobEmb}, we find a constant  $c_1 >0$ such that
\beq\label{0*} \beg{split} &|\bb_\vv(\mu)|\le c_1\big(\big\|\ss{h_\mu}\big\|_{H_{k,\aa}}^{2\theta} +1\big),\\
&\|D^E\bb_\vv(\mu)\|_{L^2(\mu)} \le \big(\|q'\|_\infty+\|\phi\|_\infty\big)\ss{\mu(\R^d)}.\end{split} \end{equation}
By \eqref{re:estimate1} and  the fact that the Gaussian measure $G_{k,\aa}$ has any finite moment in $H_{k,\aa}$, we conclude
\beq\label{*0} \int_{\M^{ac}} \mu(\R^d)^n \LL_\M(\d\mu)= G_{k,\aa}(\|\cdot\|_H^{2n})  \le \aa^{-2n} G_{k,\aa}(\|\cdot\|_{H_{k,\aa}}^{2n})  <\infty,\ \ n\in \mathbb N.\end{equation}
Thus, we may apply the dominated convergence theorem to conclude that the asymptotic for the functions in \eqref{eq:tauE} holds also in $L^2(\mac,\LL_\M).$ So, the claimed assertion is confirmed.

(b) Let $a< b\in\R^d$, $q\in C_b^1([0,\infty))$ and $\phi\in C_b(\R^d).$  Then we claim that
	\begin{equation*}
		\beta(\mu):=\int_{[a,b)}q\circ h_\mu\d\lam+\int_{[a,b)}\phi\d\mu,\quad \mu\in\mac,
	\end{equation*}
	is a member of $\D(\EE_\M)$ and
	\begin{equation*}
		 \De\beta (\mu)=1_{[a,b)}\big[q'(h_\mu)+\phi \big].	\end{equation*}

To prove this assertion, for any $n\in\mathbb N$ we denote by ${\{Q^n_j\}}_{j=1,\dots,n^d}$ the partition of $[a,b)$ into the collection of all subcubes of the type $l_1\times\dots \times l_d$
with
$$l_i=a_i+\tfrac{1}{n}[(k_i-1)(b_i-a_i),k_i(b_i-a_i)),\ \ 1\le i \le d,\ 1\le k_i\le  n.$$
Then for any $\eta\in L^1([a,b),\lam)$, it holds
\begin{equation}\label{eq:LMV}
	S_n\eta:=\sum_{j=1}^{n^d}\ff{\lam(\eins_{Q^n_j}\eta)}{\lam(Q^n_j)}\eins_{Q^n_j}\overset{n\to\infty}{\rightarrow}\eta\quad\text{in }L^1([a,b),\lam).
\end{equation}
Define for each $n\ge 1$
 \beq\label{BN1}
	\beta_n(\mu):=\int_{[a,b)}q\circ (S_n h_\mu)\d\lam+\int_{[a,b)} S_n\phi\d\mu,\quad \mu\in\mac.
\end{equation}
By step  (a) and the linearity of $\De$,    we conclude that $\beta_n\in\D(\EE_\M)$ with
\begin{equation*}
	 \De\beta_n (\mu)=1_{[a,b)}\Big[q'(S_n h_\mu(x))+ S_n\phi(x)\big],\ \ \LL_\M\text{-a.e.}\ \mu\in \M^{ac}.
\end{equation*}
  Let $\beta$ be as in the claim of (b). By \eqref{eq:LMV}, $\|q'\|_\infty<\infty$ and $ \|\eins_{[a,b)}S_n\phi\|_{\infty}\le \|\phi\|_\infty<\infty$,
we may apply the dominated convergence theorem to get
 \begin{equation}\label{eq:Sn}
	\lim_{n\to\infty}\bb_n(\mu)= \bb(\mu),\ \ \lim_{n\to\infty} \Big\|\eins_{[a,b)}\big[q'(S_{n} h_\mu)+ S_{n}\phi\big]-\eins_{[a,b)}\big[q'(h_\mu)+ \phi\big]\Big\|_{L^2(\mu)}=0,\ \ \mu\in \M^{ac}.\end{equation}
	Then by the same reason used in the end of step (a), we may further apply the dominated convergence theorem to conclude that
	\eqref{eq:Sn} holds also in   $L^2(\mac,\LL_\M)$. Then the desired assertion holds.

(c) Let $q$, $\phi$ and $\beta$ be as in the assumptions of this proposition.
For each $n\in\N$, let
$$q_n(s):=q\Big( \ff{1+ns}{n+s}\Big),\ \ \ \ \phi_n:= (-n)\lor \big[ n\land \phi \big].$$
Then $q_n\in C_b^1([0,\infty)), \phi_n\in C_b(\R^d)$, $q_n\to q$ and $\phi_n\to\phi$  as $n\to\infty$.
  By step (b), the potential
 \begin{equation}\label{BBN0}
 	\beta_n(\mu):=\int_{[-n,n)^d}\big(q_n\circ h_\mu + \phi_n h_\mu\big)\d\lll,\quad \mu\in\mac
 \end{equation}
 is a member of $\D(\EE_\M)$ and
\beq\label{DBN} \De \bb_n(\mu)= 1_{[-n,n)^d} \big[q_n'(h_\mu) +\phi_n\big].\end{equation}
By $|\phi_n|\le K(1+|\cdot|^{k})$ due to \eqref{PH},
 and  combining \eqref{BBN0}  with the first condition in  \eqref{eq:qphi}, we find a constant $c_0>0$ such that
\beg{align*} &|\bb(\mu)|\lor \sup_{n\ge 1}  |\bb_n(\mu)|\le
c_0 \lll\big(h_\mu^{\ff 1 2} + h_\mu^\theta + (1+|\cdot|^{k})h_\mu\big).\end{align*}
Thus, by
\eqref{re:estimate1},  \eqref{eq:H0L1} and \eqref{eq:SobEmb}, we find a constant $c_1>0$ such that
\beq\label{BBN1} |\bb(\mu)|\lor \sup_{n\ge 1} |\bb_n(\mu)|\le c_1\big(\big\|\ss{h_\mu}\big\|_{H_{k,\aa}}^{2\theta}+1\big).\end{equation}
 Moreover,  noting that  \eqref{eq:qphi2} implies
 $$\sup_{n\ge 1} |q_n(s)|\le c_2 (s^{-\ff 1 4}+s^{\ff{\theta-1}2}),\ \ s>0$$
 for some constant $c_2>0$,
 by \eqref{DBN} and using  this estimate in place of $\|q'\|_\infty<\infty$ which is not available in the present situation,
 the same argument enables us to   find  constants   $ c_3,c_4 >0$ such that
\beq\label{BBN2}\beg{split}&\|D^E\bb_n(\mu)\|_{L^2(\mu)}^2\lor \big\|q'(h_\mu)+ \phi\big\|_{L^2(\mu)}^2 \\
 & \le c_3 \big(\| h_\mu^{\ff 12}\|_{L^1(\lll)}^2+  \| h_\mu^{\theta}\|_{L^1(\lll)}^2+ \|(1+|\cdot|^{2k})h_\mu\|_{L^1(\lll)}\big)
 \le c_4 \big(1 +   \big\|\ss{h_\mu}\big\|_{H_{k,\aa}}^{4\theta} \big).\end{split} \end{equation}
Since
$$\int_{\mac}  \big\|\ss{\rr_\mu}\big\|_{H_{k,\aa}}^l\LL_\M(\d\mu) = \int_{H_{k,\aa}} \big\|f\big\|_{H_{k,\aa}}^lG_{k,\aa}(\d f)<\infty,\ \  l\in (1,\infty),$$   by combining \eqref{BBN0}-\eqref{BBN2}  with  $q_n(s)\to q(s)$ and $q_n'(s)\to q'(s)$, we may apply the dominated convergence theorem twice for $(\LL_\M$-a.e) $ \mu $    and  $\LL_\M$ respectively,  to conclude that
$$\lim_{n\to\infty} \Big(|\bb_n(\mu)-\bb(\mu)|^2+ 	\Big\|\De\bb_n(\mu)-\big[q'(h_\mu)+ \phi\big]\Big\|_{L^2(\mu)}^2\Big)=0,\ \ \LL_\M\text{-a.e.},$$
and furthermore,
  \begin{equation*}
 \int_{\mac} \Big(|\bb_n(\mu)-\bb(\mu)|^2+ 	\Big\|\De\bb_n(\mu)-\big[q'(h_\mu)+ \phi\big]\Big\|_{L^2(\mu)}^2\Big)\LL_\M(\d\mu)=0.
 \end{equation*}
 Then $\bb\in \D(\EE_\M)$ with
 $$\De\bb (\mu)=\big[q'(h_\mu)+ \phi\big],$$ so that
  \eqref{eq:Dbeta}  holds by definition.

 \end{proof}

By combing Propositions \ref{prp:Vqh} and \ref{thm:beta}, we have the following result.

\begin{cor}\label{cor:C2}  Let   $\theta\in \big(1,\ff d{(d-2)^+}\big]\cap (1,\infty)$ and $c,\aa_1, \aa_2\in (0,\infty)$ be   constants, and let $k,\aa \in (0,\infty) $ satisfy
\beq\label{**'} k>\ff d 2,\  \ \  \ (d+\aa)^{d'}>2\aa_1.\end{equation}
Let $\beta_{q,\phi}$ be defined in \eqref{eq:betaq} for some $q\in C^1((0,\infty))\cap C_b([0,\infty))$ and   $\phi \in C(\R^d)$ satisfying \eqref{PH} and
\beq\label{BBN} \beg{split} &-\aa_2 s^{\ff 1 2} \leq q(s)\leq c\big(s^{\ff 1 2} +s^{\theta} \big),\quad s\in[0,\infty),\\
&  \sup_{s>0 } \ff {|q'(s^2)|}{ s^{-\ff 1 2}+ s^{\theta-1} }<\infty.\end{split} \end{equation}
Let $\LL_\M=G_{k,\aa}\circ\Psi_\M^{-1}$ for $\Psi_\M(f):=f^2\lll$ and $\LL_\M^\beta:=\ff{\e^{-\beta}\LL_\M}{\LL_\M(\e^{-\beta})}.$
Then
	the diffusion process associated with the closure (see Proposition \ref{prp:Vqh} \& Corollary \ref{cor:C1}) of
	\begin{equation*}
		\EE_\M^\beta(u,v):=\int_{\mac}\langle \De u(\mu),\De v(\mu)\rangle_{L^2(\mu)}\d \LL^\beta_\M(\mu),\quad u,v\in \Ce_{b}(\M),
	\end{equation*}
	is a solution to the following SDE on $H=L^2(\R^d,\d x)$
	\begin{equation*}
		\d\mu_t= - \big[q'( h_{\mu_t})+\phi\big]\d t+\d R_t^\M,\quad t\ge 0,
	\end{equation*}
	in the  sense of Definition \ref{defn:martingaleSol}, where  $q'( h_{\mu_t})+\phi$ is well-defined as an element in $L^2(\mu)$ for $ \LL_\M$-a.e.~$\mu\in\mac$.
\end{cor}

\begin{exa}[{\bf Stochastic Entropy Extrinsic Derivative Flow}] \label{exa:entropy}   It is easy to see
	\begin{equation*}
		\lim_{s\downarrow 0}s^2\ln(s^2)=0,\quad\Big|\ff{\d}{\d  s}\big(s^2\ln(s^2)\big)\big|_{s=s_0}\Big| \leq 2\quad\text{for } s_0\in(0,1),
	\end{equation*}
	and hence $-s\ln(s)\leq 2\sqrt{s}$, $s\in(0,1)$. Defining $q(s):=s\ln(s)$, $s\in(0,\infty)$, and $q(0):=0$, it follows
	$-2\sqrt{s}\leq q(s)\leq s\ln(s\lor 1)$, $s\in[0,\infty)$ and by $q'=\ln(\cdot)+1$,
	\begin{equation*}
		\sup_{s>0 } \ff {|q'(s)|+|q'(s^2)|}{\big(s^{-\ff 1 2 }+\ln(s\lor 1)\big)}<\infty.
	\end{equation*}
	So, by Proposition \ref{thm:beta}, the entropy functional
	\begin{equation*}
		 \en(\mu):=\beg{cases} \int_{\R^d}\big(\ln(h_\mu)-1\big)\d\mu,\ &\text{if}\ \ln h_\mu\in L^1(\mu),\\
	\infty,\ &\text{otherwise}\end{cases}  \end{equation*}
	satisfies $(\beta1)$-$(\beta3), \e^{-\bb}\in L^1(\mac,\LL_\M)$ and $\en \in \D(\EE_\M)$
	with
	$$ \De (\en) (\mu)= \ln (h_\mu).$$
	Moreover,  Corollary \ref{cor:C2} provides a solution to the  following SDE on $H=L^2(\R^d,\d x)$:
	\begin{equation*}
		\d\mu_t= - \De(\en ) (\mu_t) \d t +\d R_t^\M=-  \ln (h_{\mu_t})\d t+\d R_t^\M,\quad t\ge 0,
	\end{equation*}
	in the  sense of Definition \ref{defn:martingaleSol}.
\end{exa}

\subsection{Stochastic extrinsic derivative  flows on $\pac$ over $\R^d$}

 Again  let $M=\R^d$ and $\lll $ be   the Lebesgue measure.  Let  $H_{k,\aa}$ and $G_{k,\aa}$ for constants $k,\aa>0$ are introduced in the last subsection.

  We intend to solve the following stochastic extrinsic derivative flow on $\pac$:
 \beq\label{N4} \d\mu_t = - \Det \bb(\mu_t) \d t +\d\tt R_t^\scr P,\ \ t\ge 0,\end{equation}
 where $\tt R_t$ is introduced in the second part of  Subsection 4.1 for $G= G_{k,\aa}$, so that \eqref{LLP} becomes
 \beq\label{LL0} \LL_0:= \ff{\|\cdot\|^pG_{k,\aa}}{G_{k,\aa}(\|\cdot\|^p)}.\end{equation}

 We consider the class of potentials $\bb_{q,\phi}$ defined  in \eqref{eq:betaq} for $\mu\in \pac.$ To solve \eqref{N4}, we make the following assumption.

 \beg{enumerate} \item[{\bf (A)}] Let $q\in C^1((0,\infty))\cap C([0,\infty)), \phi\in C(\R^d),$  $k>\ff d 2$ and   $\theta \in (1,\ff d{(d-2)^+}]\cap (1,\infty)$. Take $p\ge 4\theta$ in \eqref{LL0}.
  There exists a constant $c>0$ such that
  \beg{align*} &|\phi|\le c (1+|\cdot|^k),\\
		&-c\big(1+  s^{\ff 1 2}\big) \leq q(s)\leq c\big(1+ s^{\theta} \big),\quad s\in[0,\infty),\\
 & \sup_{s>0 } \ff {|q'(s^2)|}{ s^{-\ff 1 2}+ s^{\theta-1} }<\infty.
 \end{align*}
 \end{enumerate}

  \beg{thm}\label{TNN} Assume {\bf (A)}. Let $\LL_0$ be in $\eqref{LL0}.$ If $\e^{-\bb}$ is bounded on $\pac$, then the bilinear form  $(\EE_\scr P^\bb, C_b^{E,1}(\scr P))$ defined in Proposition \ref{NP'}
  for
  $$\LL_{\scr P}^\bb:= \ff{\e^{-\bb} \LL_0}{\LL_0(\e^{-\bb})}$$
  is closable in $L^2(\pac, \LL_\scr P^\bb)$, and its closure
is a local quasi-regular   Dirichlet form.  Moreover, the associated diffusion process via Corollary \ref{cor:C1} provides a martingale solution to \eqref{N4}.\end{thm}

\beg{proof}
By Proposition \ref{thm:beta}(1), {\bf (A)} implies that $\bb=\bb_{q,\phi}$ satisfies $(\bb 1)$-$(\bb 3)$. Noting that the condition on $\aa$ in \eqref{**} is  only used to prove the integrability of $\e^{-\bb}$.
So, if we further assume  that $\e^{-\bb}$ is bounded,  then by the same argument as in the proof of Proposition \ref{prp:Vqh}, but use   Proposition \ref{NP'}  in place of  Proposition \ref{NP},
we conclude  that $(\EE_\scr P^\bb, C_b^{E,1}(\scr P))$ is closable in $L^2(\pac, \LL_\scr P^\bb)$, and its closure
is a local quasi-regular  Dirichlet form.

So, by Proposition \ref{NP'}, it suffices to verify that $\bb\in \D(\EE_\scr P)$, since this  together with the boundedness of $\e^{-\bb} $ yields $\e^{-\bb}\in \D(\EE_\scr P).$
This can be done by using the approximation argument   in the proof of Proposition \ref{thm:beta},  explained as follows.

(a) Let $a< b\in\R^d$ and  $c\in \R$. It is easy to see that $\bb_{abc}$ defined for $\mu\in \scr P$ is extrinsically differentiable with
$$\Det\bb_{abc}(\mu)= \Big(1_{[a,b)} -\mu\big([a,b)\big)\Big) \Big[q'\Big(\ff{\mu([a,b))}{\lam([a,b))}\Big)+c\Big].$$
By approximating $1_{[a,b)}$ with  $\tau_\vv$ as in the proof of Proposition \ref{thm:beta}, we see that $\bb_{abc}\in \D(\EE_\scr P)$.

(b) Let $a< b\in\R^d$, $q\in C_b^1(\R)$ and $\phi\in C_b(\R^d).$  Then we claim that
	\begin{equation*}
		\beta(\mu):=\int_{[a,b)}q\circ h_\mu\d\lam+\int_{[a,b)}\phi\d\mu,\quad \mu\in\pac,
	\end{equation*}
	is a member of $\D(\EE_{\scr P})$ and
	\begin{equation*}
		 \Det\beta (\mu)=1_{[a,b)}\big[q'(h_\mu)+\phi \big]- \mu\Big(1_{[a,b)}\big[q'(h_\mu)+\phi \big]\Big).	\end{equation*}
  Indeed,    the function  $\eta_n$ defined in \eqref{BN1} can be formulated as
$$\bb_n(\mu)= \sum_{j=1}^{n^d} \bigg[\int_{Q_j^n} q\Big(\ff{\lll(1_{Q_j^n}h_\mu)}{\lll(Q_j^n)}\Big) \d\lll +\lll\big(1_{Q_j^n}\phi\big)\bigg].$$
So, by step (a),  we have $\bb_n\in  \D(\EE_{\scr P})$ and
	\begin{align*}
		& \Det\beta_n (\mu)=\De \bb_n(\mu)- \mu\big(\De \bb_n(\mu)\big),\\
		&\De \bb_n(\mu)= \sum_{j=1}^{n^d} 1_{Q_j^n} \bigg[q'\Big(\ff{\lll(1_{Q_j^n}h_\mu)}{\lll(Q_j^n)} \Big)  +\ff{\lll(1_{Q_j^n}\phi)}{\lll(Q_j^n)}  \bigg].
		 	\end{align*}
Since $\|q'\|+|\phi|$ is bounded, we have
$$ \sup_{n\ge 1, \mu\in \pac} \|\Det\beta_n (\mu)\|_\infty<\infty.$$
So, by the dominated convergence theorem, we conclude   that
 $$
	\lim_{n\to\infty}\int_{\pac} \Big(|\bb_n(\mu)- \bb(\mu)|^2+   |\Det\bb_n (\mu)-\Det \bb (\mu)\|_{L^2(\mu)}^2\Big) \d\LL_{\scr P}(\mu)=0.$$
Thus, $\bb\in \D(\EE_{\scr P}).$ 	

(c) Let $q$, $\phi$ and $\beta$ be as in the assumptions of this theorem.   Let $\bb_n$ be defined in \eqref{BBN0}.   By step (b),  we have $\beta_n\in  \D(\EE_\scr P)$ and
\beg{align*}  \Det \bb_n(\mu)= 1_{[-n,n)^d} \big[q_n'(h_\mu) +\phi_n\big]- \mu\Big(1_{[-n,n)^d} \big[q_n'(h_\mu) +\phi_n\big]\Big).\end{align*}
By {\bf (A)},  we find a constant $c_1>0$ such that
\beg{align*} & \|q_n'(h_\mu)\|_{L^2(\mu)}+ \|\phi_n\|_{L^2(\mu)}^2+ |\bb_n(\mu)|^2\\
&\le c_1 \mu\big(h_\mu^{-\ff 1 2}+ 1+|\cdot|^{2k}\big)+ c_1 \big(\mu(h_\mu^{\theta-1})\big)^2=:F(\mu).\end{align*}
Combining this with \eqref{re:estimate1}, \eqref{eq:H0L1}, \eqref{eq:SobEmb}  and $p\ge 4\theta$ in \eqref{LL0},  we find a constant $c_2>0$ such that
\beg{align*} &\LL_\scr P\big(F\big) \\
&= c_1 \int_H \bigg[\lll\bigg(\ff{f^2}{\lll(f^2)} \cdot \Big(\ff{f^2}{\lll(f^2)}\Big)^{-\ff 1 2} + \ff{f^2}{\lll(f^2)}\big(1+|\cdot|^{2k}\big)\bigg)
+\Big\{\lll\Big(\ff{f^{2\theta}}{\lll(f^2)^{\theta}}\Big\}^2\bigg]\d\LL_0(f)\\
&=\ff{c_1}{G_{k,\aa}(\|\cdot\|_H^p)}  \int_H \bigg[\ff{\lll(|f|)+\lll(f^2(1+|\cdot|^{2k}))}{\|f\|_H} + \ff{[\lll(f^{2\theta})]^2}{\|f\|_H^{4\theta}}\bigg]\|f\|_H^p\d G_{k,\aa}(f)\\
&\le c_2  \int_H \Big(\|f\|_{k,\aa}^p+\|f\|_{k,\aa}^{p+1}\Big) \d G_{k,\aa}(f)<\infty.\end{align*}
 Then, as in the proof of Proposition  \ref{thm:beta}, we may apply the dominated convergence theorem
 to get
$$\lim_{n\to\infty}  \int_{\pac} \Big(|\bb_n(\mu)-\bb(\mu)|^2+ 	\Big\|\Det \bb_n(\mu)-\big[q'(h_\mu)+ \phi- \mu(q'(h_\mu)+ \phi)\big]\Big\|_{L^2(\mu)}^2\Big)\d\LL_{\scr P}(\mu)=0,
 $$
so that  $\bb\in \D(\EE_\scr P)$ with
 $$\Det \bb (\mu)= q'(h_\mu)+ \phi - \mu\big(q'(h_\mu)+ \phi\big).$$ \end{proof}

\begin{exa} [ {\bf Stochastic Entropy-Extrinsic-Derivative Flow}] \label{4.9} Let $\phi\in C(\R^d)$ such that $|\phi|\le c(1+|\cdot|^k)$ holds for some constant $c>0$ and
$$c_\phi:= \int_{\R^d}\e^{-\phi(x)}\d x<\infty.$$
We consider the entropy functional with respect to $\mu_\phi(\d x):=\e^{-\phi(x)}\d x$:
$$\bb(\mu)=\en_{\mu_\phi}(\mu):= \int_{\R^d} \ln \big(h_\mu\e^{\phi}\big)\d\mu,\ \ \mu\in \pac.$$
Then $\bb=\bb_{q,\phi} $ for $q(s)=s\ln(s),$ and {\bf (A)} holds for any $\theta \in \big(1,\ff{d}{(d-2)^+}\big]\cap (1,\infty).$

Moreover,  by Jensen's inequality,
\beg{align*}&\bb(\mu)= c_\phi \int_{\R^d} \Big[h_\mu\e^{\phi} \ln \big(h_\mu\e^{\phi}\big)\Big]\ff{\d\mu_\phi}{c_\phi}  \\
&\ge c_\phi \bigg(\int_{\R^d} (h_\mu\e^{\phi})\ff{\d\mu_\phi}{c_\phi} \bigg)\ln \int_{\R^d} (h_\mu\e^{\phi})\ff{\d\mu_\phi}{c_\phi}\\
&= - \ln c_\phi >-\infty.\end{align*}
Thus, $\e^{-\bb}$ is bounded, so that by Theorem \ref{TNN},  the following SDE on $\pac$ has a martingale solution:
$$\d\mu_t = -\big(\Det\en_{\mu_\phi}\big) (\mu_t) \d t+ \d\tt R_t^\scr P,\ \ t\ge 0.$$
 \end{exa}


\begin{thebibliography}{999}



\bibitem{AKR}Albeverio, S., Kondratiev, Y., R\"ockner, M. Differential geometry of Poisson spaces. {\em C. R. Acad. Sci. Paris S\'er. I Math.} \textbf{323}, 1129-1134 (1996)
\bibitem{AM}Albeverio, S., Ma, Z. A general correspondence between Dirichlet forms and right processes. {\em Bull. Amer. Math. Soc.}  \textbf{26}, 245-252 (1992)
\bibitem{AMR15}Albeverio, S., Ma, Z., R\"ockner, M.   Quasi regular Dirichlet forms and the stochastic quantization problem. In: Festschrift Masatoshi Fukushima, Chapter 3,
27-58, World Scientific, 2015
\bibitem{AR}Albeverio, S., R\"ockner, M. Classical Dirichlet forms on topological vector spaces $-$ Closability and a Cameron-Martin formula. {\em J. Funct. Anal.}  \textbf{88}, 395-436 (1990)
\bibitem{BRW}Bao, J., Ren, P., Wang, F. Bismut formula for Lions derivative of distribution-path dependent SDEs. {\em J. Differ. Equ.} \textbf{282}, 285-329 (2021)
\bibitem{BK10}Blount, D., Kouritzin, M. On convergence determining and separating classes of functions. {\em Stoch. Proc. Appl.} \textbf{120}, 1898-1907 (2010)
\bibitem{BH91}Bouleau, N., Hirsch, F. Dirichlet Forms and Analysis on Wiener Space. Walter de Gruyter, Berlin, 1991
\bibitem{CMR}Chen, Z., Ma, Z., R\"ockner, M. Quasi-homeomorphisms of Dirichlet forms. {\em Nagoya Math. J.} \textbf{136}, 1-15 (1994)
\bibitem{DZ} Da Prato, G., Zabczyk, J. Second Order Partial Differential Equations in Hilbert Spaces. Cambridge University Press, 2002
\bibitem{Sch}Dello Schiavo, L. The Dirichlet-Ferguson diffusion on the space of probability measures over a closed Riemannian manifold. {\em Ann. Probab.} \textbf{50}, 591-648 (2022)
\bibitem{Eb} Eberle, A. Girsanov-type transformations of local Dirichlet forms: an analytic approach. {\em Osaka J. Math.} \textbf{33}, 497-531 (1996)
\bibitem{Fu}Fukushima, M. Dirichlet spaces and strong Markov processes. {\em Trans. Amer. Math. Soc.}  \textbf{162}, 185-224 (1971)
\bibitem{FOT11}Fukushima, M., Oshima, Y.~Takeda, M. Dirichlet forms and symmetric Markov processes. Walter de Gruyter \& Co., Berlin, 2011
\bibitem{HKPS}Hida, T., Kuo, H., Potthoff, J., Streit, L. White Noise. Springer Science+Business Media, Dordrecht, 1993
\bibitem{KLV}Kondratiev, Y., Lytvynov, E.~Vershik, A. Laplace operators on the cone of Radon measures. {\em J. Funct. Anal.} \textbf{269}, 2947-2976 (2015)
\bibitem{MR92}Ma, Z., R\"ockner, M. Introduction to the theory of (nonsymmetric) Dirichlet forms. Springer-Verlag, Berlin, 1992
\bibitem{Maz}Maz'ya, V. Sobolev Spaces. 2nd edition. Springer, Heidelberg, 2011.
\bibitem{ORS}Overbeck, L., R\"ockner, M., Schmuland, B. An analytic approach to Fleming-Viot processes with interactive selection. {\em Ann. Probab.}  \textbf{23}, 1-36 (1995)
\bibitem{RW20}Ren, P., Wang, F.-Y. Spectral gap for measure-valued diffusion processes. {\em J. Math. Anal. Appl.} \textbf{483}, 123624 (2020)
\bibitem{RW21}Ren, P.,   Wang, F.-Y., Derivative formulas in measure on Riemannian manifolds. {\em Bull. London Math. Soc.} \textbf{53}, 17861800 (2021)
\bibitem{RW22}Ren, P., Wang, F.-Y. Ornstein-Uhlenbeck type processes on Wasserstein spaces. {\em Stoch. Proc. Appl.} \textbf{172}, 104339 (2024)
\bibitem{RWW24} Ren, P.,   Wang, F.-Y., Wittmann, S., Diffusion Processes  on  $p$-Wasserstein Space over Banach Space, Preprint,    arXiv:2402.15130.
\bibitem{Sturm}Renesse, M., Sturm, K.,  Entropic measure and Wasserstein diffusion. {\em Ann. Probab.}\textbf{37}, 1114-1191 (2009)
\bibitem{ST} Sturm, K., Wasserstein diffusion on multidimensional spaces, {\em arXiv:2401.12721}
\bibitem{Shao}Shao, J. A new probability measure-valued stochastic process with Ferguson-Dirichlet process as reversible measure. {\em Electron. J. Probab.} \textbf{16} pp. no. 9, 271-292 (2011)

\end{thebibliography}
\end{document}